\documentclass{article}

\usepackage{amsmath,amsfonts,latexsym,amssymb,amsthm,xurl}

\usepackage[utf8]{inputenc}

\DeclareMathOperator{\Gal}{Gal}

\DeclareMathOperator{\height}{h}

\newcommand{\C}{{\mathbb C}}

\newcommand{\K}{{\mathbb K}}
\renewcommand{\L}{{\mathbb L}}
\newcommand{\M}{{\mathbb M}}
\newcommand{\Q}{{\mathbb Q}}
\newcommand{\R}{{\mathbb R}}
\newcommand{\Z}{{\mathbb Z}}

\newcommand{\calA}{\mathcal{A}}
\newcommand{\DD}{\mathcal{D}}
\newcommand{\NN}{\mathcal{N}}
\newcommand{\OO}{\mathcal{O}}
\newcommand{\PP}{\mathcal{P}}

\newcommand{\gerp}{\mathfrak{p}}
\newcommand{\Gerp}{\mathfrak{P}}

\newcommand\ab{{\mathrm{ab}}}
\newcommand\const{{\mathrm{const}}}
\newcommand\tho{{\text{th}}}

\DeclareMathOperator\ord{ord}

\newtheorem{theorem}{Theorem}[section]
\newtheorem{proposition}[theorem]{Proposition}
\newtheorem{corollary}[theorem]{Corollary}
\newtheorem{lemma}[theorem]{Lemma}
\newtheorem{remark}[theorem]{Remark}

\newtheorem{question}[theorem]{Question}
\newtheorem{example}[theorem]{Example}

\numberwithin{equation}{section}

\title{Twisted rational zeros of\\ linear recurrence sequences}

\author{Yuri Bilu\\
IMB, Université de Bordeaux \& CNRS\\
E-mail: yuri@math.u-bordeaux.fr,
\and
Florian Luca\\
School of Maths, Wits, South Africa\\
CCM, UNAM, Morelia, Mexico\\
E-mail: Florian.Luca@wits.ac.za,
\and
Joris Nieuwveld and Jo\"el Ouaknine\\
Max Planck Institute for Software Systems, Germany\\
E-mail: jnieuwve@mpi-sws.org; joel@mpi-sws.org,
\and
James Worrell\\
Department of Computer Science, Oxford University, UK\\
Email: jbw@cs.ox.ac.uk
}

\date{version of \today}

\makeatletter

\renewcommand*\l@section[2]{
	\ifnum \c@tocdepth >\z@
	\addpenalty\@secpenalty
	\addvspace{0.2em \@plus\p@}
	\setlength\@tempdima{1.5em}
	\begingroup
	\parindent \z@ \rightskip \@pnumwidth
	\parfillskip -\@pnumwidth
	\leavevmode \bfseries
	\advance\leftskip\@tempdima
	\hskip -\leftskip
	#1\nobreak\hfil \nobreak\hb@xt@\@pnumwidth{\hss #2}\par
	\endgroup
	\fi}

\makeatother

\begin{document}

\hfuzz5pt

	\maketitle

\begin{abstract}
 We introduce the notion of a twisted rational zero of a
 non-degenerate linear recurrence sequence (LRS\@). We show that any
 non-degenerate LRS has only finitely many such twisted rational zeros. In the particular case of the Tribonacci sequence, we show that $1/3$ and $-5/3$ are the only twisted rational zeros which are not integral zeros. 
\end{abstract}

	{\footnotesize
		
		\tableofcontents
		
	}

\section{Introduction}
Let~$\K$ be a field of characteristic~$0$. We fix an algebraic closure~$\overline{\K}$. 
By a \textit{$\K$-valued linear recurrence sequence} (LRS) of order~$m$ we mean a map ${U:\Z\to \K}$ such that for every ${n\in\Z}$ we have 
\begin{equation}
\label{erec}
U(n+m)=a_{m-1}U(n+m-1)+\cdots+ a_0U(n), 
\end{equation}
where ${a_0, \ldots ,a_{m-1}\in \K}$, with ${a_0\ne 0}$. Sometimes
instead of ``$\K$-valued LRS'' we will say ``LRS over~$\K$''.

\subsection{Twisted zeros and $p$-adic orders}

Our initial motivation was the work of Marques  and Lengyel~\cite{ML14}, who computed the $2$-adic order of the $n^\tho$ Tribonacci number $T(n)$. The \textit{Tribonacci numbers} is the $\Q$-valued LRS of order~$3$, defined by 
$$
T(0)=0, \quad T(1)=T(2)=1, \quad T(n+3)=T(n+2)+T(n+1)+T(n). 
$$
Marques  and Lengyel proved that 
\begin{equation}
\label{eml}
\nu_2(T(n))=
\begin{cases}
0, & \text{if  $n\equiv 1,2\pmod 4$};\\
1, & \text{if $n\equiv 3,11\pmod {16}$};\\
2, & \text{if $n\equiv 4,8\pmod {16}$};\\
3, &\text{if $n\equiv 7\pmod {16}$};\\
\nu_2(n)-1, & \text{if $n\equiv 0\pmod {16}$};\\
\nu_2(n+4)-1, &\text{if $n\equiv 12\pmod {16}$};\\
\nu_2(n+17)+1, &\text{if $n\equiv 15\pmod {32}$};\\
\nu_2(n+1)+1, &\text{if $n\equiv 31\pmod {32}$}.
\end{cases}
\end{equation}
They conjectured that similar formulas must hold for other primes, not just for ${p=2}$. This conjecture was refuted in~\cite{BLNOW23}, where it is shown that having formulas like~\eqref{eml} is exceptional rather than typical.

As one can see in~\eqref{eml}, on certain residue classes one has  formulas like ${\nu_2(T(n)) =\nu_2(n-a)+\const}$, where~$a$ is one of the numbers ${0,-1,-4,-17}$. This not surprising, because these numbers  are exactly the zeros of~$T$: for ${n\in \Z}$ we have
$$
T(n)=0  \quad \text{if and only if} \quad n\in \{0,-1,-4,-17\}. 
$$  
(For a proof see,  for instance,~\cite{MT91}, Example~2 on page~360; in that example~$u_n$  corresponds to our $T(-n)$.) 

Let~$U$ be a $\Q$-valued LRS; in particular, the coefficients ${a_0, \ldots, a_{m-1}}$ of the recurrence relation~\eqref{erec} belong to~$\Q$. 
We call~$p$ a \textit{regular prime} for~$U$ if it does not divide the denominators of the rational numbers ${a_0, \ldots, a_{m-1}}$, and the numerator of~$a_0$. In other words, ${a_0, \ldots, a_{m-1}}$ are $p$-adic integers and~$a_0$ is a $p$-adic unit. 

The following theorem can be easily proved,  using $p$-adic analysis, see  Section~\ref{sorder}. 

\begin{theorem}
\label{theasy}
Let~$a$ be a zero of a $\Q$-valued LRS~$U$, and~$p$ a regular prime for~$U$.  Then there exist a positive integer~$Q$, a positive integer~$\kappa$ and an integer~$\tau$  such that 
\begin{equation}
\label{econdaa}
\nu_p(U(n))= \kappa \nu_p(n-a)+\tau \qquad\text{when}\quad n\equiv a\pmod Q
\end{equation} 
\end{theorem}

Of course the converse also holds: if for some~$p$ there exist $Q,\kappa,\tau$ as above such that~\eqref{econdaa} holds,  then~$a$ is a zero of~$U$.



Now let us ask the following slightly more general question: what would happen if we take~$n$ in~\eqref{econdaa} not from the residue class of~$a$ modulo~$Q$, but from a different residue class? 
\begin{question}
\label{quaaprime}
Let~$p$ be a prime number. Assume that there exist 
\begin{equation*}
Q\in \Z_{>0}, \qquad a'\in \{0,1, \ldots, Q-1\}, \qquad \kappa \in \Z_{>0}, \qquad \tau \in \Z
\end{equation*}
such that 
\begin{equation*}
\nu_p(U(n))= \kappa \nu_p(n-a)+\tau \qquad\text{when}\quad n\equiv a'\pmod Q. 
\end{equation*} 
Assume further that  $\nu_p(n-a)$ is not bounded on the residue class of $a'$ modulo~$Q$. 
Does it imply that ${a'\equiv a\pmod Q}$ and ${U(a)=0}$? 
\end{question}
As the result of Marques  and Lengyel implies, it is indeed the case when ${U=T}$ and ${p=2}$. But is it true in general?

The answer is ``no'', as the following example shows.  

\begin{example}
\label{extwonplusone}
Consider ${U(n):=2^n+1}$ and let~$p$ be a prime number satisfying ${p\equiv \pm3\pmod 8}$. Then ${2^{(p-1)/2}+1\equiv0\pmod p}$. Define
${\tau:=\nu_p(2^{(p-1)/2}+1)}$. Then 
$$
\nu_p(U(n))=\nu_p(n)+\tau \quad \text{when} \quad n\equiv \frac{p-1}2\pmod {p-1}. 
$$
However, ${U(0)\ne 0}$, and in fact $U(n)$ does not vanish at all. 
\end{example}

The explanation is that~$0$ is, in fact, a kind of ``hidden'' zero of the LRS ${2^n+1}$. To give an exact definition, recall that a $\K$-valued LRS~$U$ satisfying~\eqref{erec} admits the \textit{Binet expansion}  
\begin{equation}
\label{ebinexp}
U(n)=f_1(n)\lambda_1^n + \cdots +f_s(n)\lambda_s^n\qquad (n\in \Z), 
\end{equation}
where ${\lambda_1, \ldots, \lambda_s}$ are the distinct roots of the characteristic polynomial 
$$
X^m-a_{m-1}X^{m-1}- \cdots -a_0, 
$$ 
and ${f_1, \ldots, f_s}$ are polynomials with  coefficients in the field $\K(\lambda_1, \ldots, \lambda_s)$. (Note that the roots $\lambda_i$ are non-zero, because ${a_0\ne 0}$.)  We call ${a\in \Z}$ a \textit{twisted zero} of the $\K$-valued LRS~$U$ if there exist roots of unity ${\xi_1, \ldots, \xi_s\in \overline{\K}}$ such that 
$$
\xi_1f_1(a)\lambda_1^a + \cdots +\xi_sf_s(a)\lambda_s^a=0. 
$$
For example,~$0$ is a twisted zero of the LRS with general term ${2^n+1^n}$, because 
$$
1\cdot 2^0+ (-1)\cdot1^0=0.
$$  

In Section~\ref{sorder} we will prove the following theorem, which gives a partial positive answer to Question~\ref{quaaprime}.  

\begin{theorem}
\label{thtwint}
Let~$U$ be a $\Q$-valued LRS, ${a\in \Z}$ and~$p$ a regular prime number for~$U$. 
Assume that there exists a  sequence of integers  $(n_k)$ satisfying
$$
\nu_p(U(n_k))\to +\infty, \qquad \nu_p(n_k-a) \to +\infty. 
$$
Then~$a$ is a twisted zero of~$U$. 
\end{theorem}

There is another phenomenon discovered in~\cite{BLNOW23}, again in the context of Tribonacci numbers. There exist infinitely many prime numbers~$p$ 
such that ${\nu_p(T(n))\ge \nu_p(n-1/3)}$ for  ${n\equiv1/3\pmod{p-1}}$, and the same holds true with $1/3$ replaced by $-5/3$;  see \cite[Theorem~1.5]{BLNOW23}. 
The reason is that $1/3$ and $-5/3$ can be viewed as ``rational zeros'' of the LRS~$T$, see \cite[Section~2]{BLNOW23}. 

Let us give the general definition. We call ${a\in \Q}$ a \textit{rational zero} of the $\K$-valued LRS~$U$ with Binet expansion~\eqref{ebinexp} if, for some definition of the rational powers ${\lambda_1^a, \ldots, \lambda_s^a\in \overline{\K}}$,  we have 
$$
f_1(a)\lambda_1^a + \cdots +f_s(a)\lambda_s^a=0. 
$$
We call~$a$ a \textit{twisted rational zero} (TRZ) of~$U$ if, for some definition of  ${\lambda_1^a, \ldots, \lambda_s^a}$ and some roots of unity ${\xi_1, \ldots, \xi_s}$,   we have 
$$
\xi_1f_1(a)\lambda_1^a + \cdots +\xi_sf_s(a)\lambda_s^a=0. 
$$

Theorem~\ref{thtwint} remains true assuming that ${a\in \Q}$.

\begin{theorem}
\label{thtwrat}
Let~$U$ be a $\Q$-valued LRS, ${a\in \Q}$ and~$p$ a regular prime number for~$U$. 
Assume that there exists a  sequence of integers  $(n_k)$ satisfying
\begin{equation}
\label{enk}
\nu_p(U(n_k))\to +\infty, \qquad \nu_p(n_k-a) \to +\infty. 
\end{equation}
Then~$a$ is a TRZ of~$U$. 
\end{theorem}
This theorem is proved in Section~\ref{sorder} as well. 

One may ask whether the converse is true; that is, if~$a$ is a TRZ, then there exists a sequence of integers $(n_k)$ satisfying~\eqref{enk}. 
Easy examples show that the answer is ``no'' in general. 
\begin{example}
\label{expminone}
If ${p\equiv -1\pmod 8}$ then
${\nu_p(2^n+1)=0}$ for all~$n$, though~$0$ is a twisted zero of the LRS  ${2^n+1^n}$. 
\end{example}

One may still hope that, when~$a$ is a TRZ,
this holds for infinitely many primes.


\begin{question}
\label{quhard}
Let~$a$ be a TRZ of a non-degenerate $\Q$-valued LRS~$U$. Are there  infinitely many prime numbers~$p$ with the following property: there exists a sequence of integers $(n_k)$ satisfying~\eqref{enk}? 
\end{question}

We show that the answer is ``yes''  for twisted \textbf{integral} zeros of  LRS of order~$2$; in fact, we will show that for them an analogue of Theorem~\ref{theasy} holds. 
However, we do not know the answer for rational zeros.  As for  LRS of higher order, the answer is, in general, ``no'' even for integral twisted zeros. See Section~\ref{squest} for the details.

\subsection{Finiteness}

Call a non-zero LRS with roots ${\lambda_1, \ldots, \lambda_s}$ \textit{non-degenerate} if $\lambda_k/\lambda_\ell$ is not a root of unity for ${k\ne \ell}$. The following  statement is the classical \textit{Skolem-Mahler-Lech Theorem}.
\begin{theorem}
[Skolem-Mahler-Lech]
\label{thsml}
A non-degenerate  linear recurrence sequence $U$ over a field of characteristic~$0$ has at most finitely many zeros:
$$
\#\{n\in \Z: U(n)=0\}<\infty.
$$
\end{theorem}
In Section~\ref{sfin} we prove that 
the same holds true for TRZs.

\begin{theorem}
\label{thfin}
Let~$U$ be a non-degenerate linear recurrence sequence  with values in a field of characteristic zero.  
Then~$U$ admits at most finitely many TRZs. 
\end{theorem}

The proof is a variation of the principal argument of Laurent's article~\cite{La84}. The main step  is bounding the denominators of  the TRZs; moreover, the bound is effective if~$\K$ is a number field. After the denominators are bounded, Theorem~\ref{thfin} can be reduced to the Skolem-Mahler-Lech theorem using the existing results about equations in roots of unity \cite{CJ76,DZ00,Ma65}.


The Skolem-Mahler-Lech Theorem is, in general, non-effective, and so is our Theorem~\ref{thfin}: while we bound effectively the denominators of the TRZs, we  cannot do the same for their numerators. However, the Skolem-Mahler-Lech Theorem can be made effective in many special cases, and so can be  Theorem~\ref{thfin}. To illustrate this, we prove (see Section~\ref{stribo}) the following.

\begin{theorem}
\label{thtribo}
The only TRZs of the Tribonacci sequence~$T$ are 
$$
0,-1,-4,-17, \frac13, -\frac53. 
$$
\end{theorem}

\section{Preliminaries}

In this section we collect some basic facts and conventions that will be used throughout the article, usually without special reference. 

\subsection{Fields}

The letter~$p$ denotes a prime number, and blackboard boldface letters $\K,\L,\M$ etc.\ denote (unless indicated otherwise) fields of characteristic~$0$. In particular, they can be number fields or local fields (finite extensions of~$\Q_p$). If~$\K$ is a number field and~$\gerp$ is a prime of~$\K$ then~$\K_\gerp$ denotes the $\gerp$-adic completion.

For every positive integer~$m$ we fix a primitive root of unity of order~$m$ and denote it~$\zeta_m$.  We denote by~$\mu_m$ the group of roots of unity of order~$m$. Given a field~$\K$, we denote by $\mu_\K$ the group of roots of unity belonging to~$\K$.

The following lemma, which is Theorem~9.1  in~\cite[Chapter~VI]{La02}, will be used in the article on several occasions. 

\begin{lemma}
\label{llang}
Let~$\K$ be a field of characteristic~$0$  and  ${\alpha\in \K^\times}$. Let~$m$ be a positive integer. Assume 
that 
\begin{align}
\label{efirstass}
&\text{for all  ${p\mid m}$ we have ${\alpha\notin \K^p}$}, \\
\label{esecondass}
&\text{when 
 ${4\mid m}$ we have ${\alpha\notin -4\K^4}$}. 
\end{align}
Then the polynomial ${X^m - \alpha}$ is irreducible in $\K[X]$. 
\end{lemma}

\begin{remark}
\label{rlang}
If ${\sqrt{-1}\in \K}$ then assumption~\eqref{esecondass} can be omitted,   because in this case  ${-4\in \K^4}$ and~\eqref{esecondass} follows from~\eqref{efirstass}.  
\end{remark}

\subsection{Linear recurrence sequences} 
\label{ssgen}

Let~$U$ be an LRS with values in a field~$\K$. We call~$m$ the \textit{minimal order} of~$U$ if~$U$ admits a linear recurrence relation of order~$m$, but not of order strictly smaller than~$m$. By convention, the minimal order of the identically zero LRS is set to be~$0$. 

Let~$m$ be the minimal order of a (non-zero) LRS~$U$ with values in~$\K$. Then the coefficients ${a_0, \ldots, a_{m-1}}$ of the recurrence relation 
\begin{equation*}
U(n+m)=a_{m-1}U(n+m-1)+\cdots+ a_0U(n)
\end{equation*}
are well-defined and belong to the field~$\K$. Fix an algebraic closure~$\overline{\K}$, and let ${\lambda_1, \ldots, \lambda_s\in \overline{\K}}$ be the distinct roots of the characteristic polynomial 
\begin{equation}
\label{echarpol}
X^m-a_{m-1}X^{m-1}-\cdots-a_0.  
\end{equation}
Then~$U$ admits the Binet expansion 
\begin{equation*}
U(n)=f_1(n)\lambda_1^n + \cdots +f_s(n)\lambda_s^n,  
\end{equation*}
where ${f_1, \ldots, f_s}$ are polynomials with  coefficients in the field $\K(\lambda_1, \ldots, \lambda_s)$, such that the order of~$\lambda_i$ as a root of the characteristic polynomial~\eqref{echarpol} is equal to ${\deg f_i+1}$. In particular, the polynomials~$f_i$ are all non-zero, and 
$$
\sum_{i=1}^s(\deg f_i+1)=m. 
$$

Unless the contrary is stated explicitly, in this article, when referring to an LRS of order~$m$, we will assume that~$m$ is the minimal order of this LRS\@.

It is important to note the following: if~$U$ is non-degenerate then it does not vanish identically on any residue class; that is, for any positive integer~$N$ and any ${\ell\in \{0, \ldots, N-1\}}$, the function ${n\mapsto U(\ell+Nn)}$ is not identically zero. Indeed, assuming non-degeneracy of~$U$, the numbers $\lambda_1^N, \ldots, \lambda_s^N$ are all distinct. Hence 
$$
U(\ell+Nn)= \sum_{i=1}^s h_i(n)\theta_i^n, \quad \text{where}\quad h_i(T):=\lambda_i^\ell f_i(\ell+NT), \quad \theta_i:=\lambda_i^N.  
$$
This implies that $U(\ell+Nn)$ is an LRS of the same minimal order as~$U$; in particular, it is not identically zero.

\section{Twisted rational zeros and the $p$-adic order}
\label{sorder}

In this section we prove Theorems~\ref{theasy},~\ref{thtwint} and~\ref{thtwrat} from the Introduction. The proofs rely on Skolem's  $p$-adic interpolation of LRS, briefly recalled in Subsections~\ref{sspadic} and~\ref{ssfunction}.

\subsection{$p$-adic analytic functions}
\label{sspadic}
In this subsection we recall some very basic facts about $p$-adic analytic functions. Most of them are quite standard. All missing proofs, unless indicated otherwise, can be found in any standard text like~\cite{Go20} or~\cite{Sc06}.

Let~$p$ be a prime number. We fix an algebraic closure $\overline{\Q_p}$,  and extend the standard $p$-adic absolute value ${|\cdot|_p}$ to~$\overline{\Q_p}$, so that ${|p|_p=p^{-1}}$. We will also use the additive valuation $\nu_p$ defined by ${\nu_p(z)= - \log|z|_p/\log p}$ for ${z\in \overline{\Q_p}^\times}$, with the convention ${\nu_p(0)=+\infty}$. All algebraic extensions of~$\Q_p$ occurring below will be viewed as subfields of this fixed $\overline{\Q_p}$.

Let~$\K$ be a finite extension of~$\Q_p$. 
For ${a\in \K}$ and ${r>0}$ we denote $\DD(a,r)$ and $\overline\DD(a,r)$ (or $\DD_\K(a,r)$, $\overline{\DD_\K}(a,r)$, if we want to indicate that the disk is in the field~$\K$) the open and the closed disks with center~$a$ and radius~$r$:
$$
\DD(a,r)=\{z \in \K: |z-a|_p<r\}, \qquad \overline \DD(a,r)=\{z \in \K: |z-a|_p\le r\}.
$$ 
It might be worth noting that every open disk in~$\K$ is also closed, and any closed disk is open. That is, for every ${r>0}$ there exist ${r',r''>0}$ such that 
$$
\DD(a,r)=\overline{\DD}(a,r'), \qquad \overline{\DD}(a,r)=\DD(a,r''). 
$$
Another useful observation is that every point of a disk serves as its center: if ${b\in \DD(a,r)}$ then ${\DD(a,r)=\DD(b,r)}$, and similarly for the closed disks.

We denote by $\OO_\K$, or simply by~$\OO$ if this does not lead to a confusion, the ring of integers of~$\K$:
$$
\OO=\{z\in \K: |z|_p \le 1\} = \overline\DD(0,1).  
$$

Let~$D$ be a disk in~$\K$ (open or closed), and~$\L$ a finite extension of~$\K$. We call ${g:D\to \L}$  an analytic function if for some ${a\in D}$ we have 
\begin{equation}
\label{eanf}
g(z) =\sum_{n=0}^\infty \alpha_n (z-a)^n \qquad (z\in D), 
\end{equation}
where ${\alpha_0, \alpha_1, \alpha_2, \ldots \in \L}$. In particular, the infinite sum on the right converges for all ${z\in D}$. 

Here are some simple properties of analytic functions, to be used below without special reference. 

\begin{enumerate}
\item
The coefficients ${\alpha_0, \alpha_1, \alpha_2, \ldots}$ are well-defined as soon as~$g$ and~$a$ are given. In particular, if the coefficients are not all~$0$, then~$g$ is a non-zero function. 

\item
The analytic function~$g$ admits a power series expansion around any other ${b\in D}$. Specifically, 
for any ${b\in \OO}$ we have  
\begin{equation}
\label{equb}
g(z) = \sum_{k=0}^\infty \beta_k(z-b)^k, 
\end{equation} 
where 
$$ 
\beta_k=\frac{g^{(k)}(b)}{k!}= \sum_{n=k}^\infty \binom nk \alpha_n (b-a)^{n-k}.
$$

\item
An analytic function on~$D$ is bounded. Indeed, set 
$$
r:=\max\{|z-w|_p: z,w\in D\}.  
$$
Then ${D=\overline{\DD}(a,r)}$, and convergence in~\eqref{eanf} is equivalent to ${|\alpha_n|_pr^n\to0}$. In particular, the sequence $|\alpha_n|_pr^n$ is bounded. Hence $\alpha_n (z-a)^n$ is bounded uniformly in ${z\in D}$. It follows that~$g$ is bounded. 

\item
A non-zero analytic function on a disk~$D$ may have at most finitely many zeros in~$D$; this is because~$D$ is compact and the zeros are isolated. A quantitative version is given by the classical Theorem of Strassmann; see, for instance, \cite[Theorem~4.1]{Ca86}. 
\end{enumerate}

\begin{theorem}
\label{thfunml}
Let~$\K$ be a finite extension of~$\Q_p$ of ramification index~$e$, and let ${g:\Z_p\to \K}$ be an analytic function, not identically~$0$. Denote by~$\calA$ the (finite) set of zeros of~$g$.  
Then there exists a positive integer~$k$ such that  for every  ${i \in \{0,1,\ldots,  p^k-1\}}$  we have one of the following two options.
\begin{enumerate}
\item[(C)]
There exists ${\tau_i\in \Z}$ such that for ${z\in \Z_p}$  satisfying ${z\equiv i \pmod {p^k}}$ we have ${\nu_p(g(z)) =e^{-1}\tau_i}$; in other words, $\nu_p(g(z))$ is constant on the residue class ${z\equiv i\pmod  {p^k}}$. 

\item[(L)]
There exist 
$$
a_i\in \calA, \qquad \tau_i\in \Z, \qquad  \kappa_i \in \Z_{>0}
$$ 
such that ${a_i\equiv i\pmod{p^k}}$, and for   ${z\in \Z_p}$  satisfying ${z\equiv i \pmod {p^k}}$ we have 
$$
\nu_p(g(z)) =\kappa_i\nu_p(z-a_i)+e^{-1}\tau_i. 
$$
\end{enumerate}

\end{theorem}

\begin{proof}
This is Theorem~3.2 from~\cite{BLNOW23}. 
\end{proof}

We denote 
\begin{equation}
\label{erho}
\rho:=p^{-1/(p-1)}.
\end{equation} 
Let us recall the definition and the basic properties of the $p$-adic exponential and logarithmic function. 

\begin{enumerate}
\item
For ${z\in \DD(0,\rho)}$ we define  
$$
\exp(z) =\sum_{n=0}^\infty \frac{z^n}{n!}. 
$$
For ${z,w\in \DD(0,\rho)}$ we have
$$
|\exp(z)-1|_p=|z|_p, \quad \exp(z+w)=\exp(z)\exp(w), \quad \exp'(z)=\exp(z). 
$$
\item
For ${z\in \DD(1,1)}$ we define  
$$
\log(z) =\sum_{n=1}^\infty \frac{(-1)^{n-1}(z-1)^n}{n}. 
$$
For ${z,w\in \DD(1,1)}$ we have
$$
\log(zw)=\log(z)+\log(w), \quad \log'(z)=\frac1z. 
$$
\item
For ${z\in \DD(1,\rho)}$ we have 
$$
|\log(z)|_p=|z-1|_p, \quad \exp(\log(z)) =z. 
$$
\item
For ${z\in \DD(0,\rho)}$  we have 
${\log(\exp(z))=z}$. 
\end{enumerate}

\subsection{$p$-adic analytic interpolation of a linear recurrence sequence}
\label{ssfunction}
The contents of this subsection is very classical and goes back to Skolem. Still, we prefer to include some proofs for the reader's convenience.

Let~$U$ be a non-zero LRS of (minimal) order~$m$ with values in a number field~$\K$. We write its recurrence relation as 
\begin{equation*}
U(n+m)=a_{m-1}U(n+m-1)+\cdots+ a_0U(n), 
\end{equation*}
where ${a_0, \ldots, a_{m-1}\in \K}$. 


We call a prime~$\gerp$ of~$\K$  \textit{regular} for~$U$ if ${a_1, \ldots, a_m}$ are $\gerp$-adic integers and~$a_0$ is a $\gerp$-adic unit.

Let~$\gerp$ be a regular prime for~$U$, and ~$\K_\gerp$  the $\gerp$-adic completion of~$\K$. It is a finite extension of~$\Q_p$, where~$p$ is the rational prime number below~$\gerp$.


\begin{proposition}
[Skolem]
\label{prsko}
There exists a positive integer~$N$ and analytic functions 
$$
g_0, \ldots, g_{N-1}:\Z_p\to \OO_{\K_\gerp}
$$
such that  
\begin{equation}
\label{eextrap}
u(\ell+Nn)= g_\ell(n) \qquad (\ell\in \{0, \ldots,N-1\}, \quad n\in \Z).
\end{equation}
Moreover, if the LRS~$U$ is non-degenerate, then none of the  functions $g_\ell$ vanishes identically. 
\end{proposition}

\begin{proof}
Denote by~$\L$ the splitting field over~$\K_\gerp$ of the characteristic polynomial ${X^m-a_{m-1}X^{m-1}-\cdots-a_0}$. Then~$U$ admits the Binet expansion 
\begin{equation*}
U(n)=f_1(n)\lambda_1^n + \cdots +f_s(n)\lambda_s^n,     
\end{equation*}
where ${\lambda_1, \ldots, \lambda_s}$ are the distinct roots of the characteristic polynomial,  and  ${f_1, \ldots, f_s}$ are non-zero polynomials with coefficients in~$\L$. Since~$p$ is a regular prime, ${\lambda_1, \ldots, \lambda_s \in \OO_\L^\times}$. 

Let~$\rho$ be defined as in~\eqref{erho}. Since ${\rho<1}$, the disk $\DD_\L(1,\rho)$ is a multiplicative group. It is a finite index subgroup of~$\OO_\L^\times$, because~$\OO_\L^\times$ is compact and   $\DD_\L(1,\rho)$ is open. Hence there exists a positive integer~$N$ such that ${x^N\in \DD_\L(1, \rho)}$ for every ${x \in \OO_\L^\times}$. Note that we have 
\begin{equation}
\label{eexp}
x^{Nn}=\exp(n\log(x^N))\qquad (x\in \OO_\L^\times, \quad n\in \Z). 
\end{equation}
Now define, for ${\ell=0,1, \ldots, N-1}$ and ${z\in \Z_p}$,  
$$
g_\ell(z) := \sum_{i=1}^s\lambda_i^\ell f_i(\ell+Nz) \exp(z\log(\lambda_i^N)). 
$$
Note that, \textit{a priori}, ${g_\ell(z)\in \L}$, but we will see that ${g_\ell(z)\in \K_\gerp}$ in a while. 

From~\eqref{eexp} we deduce that~\eqref{eextrap} holds.
In particular, ${g_\ell(z)\in \K}$  when ${z\in \Z}$. Since~$\Z$ is dense in~$\Z_p$, we have ${g_\ell(z)\in \K_\gerp}$  for all ${z\in \Z_p}$.

When~$U$ is non-degenerate, the function~$g_\ell$ does not vanish identically, because the function ${n\mapsto U(\ell+Nn)}$ does not, see Subsection~\ref{ssgen}.  
\end{proof}

As a by-product, we prove the Theorem of Skolem-Mahler-Lech (see Theorem~\ref{thsml} above) for $\overline{\Q}$-valued LRS\@. 

\begin{corollary}
[Skolem-Mahler-Lech]
Let~$U$ be a non-degenerate $\overline{\Q}$-valued LRS\@. Then the equation ${U(n)=0}$ has at most finitely many solutions in ${n\in \Z}$. 
\end{corollary}

\begin{proof}
Pick some prime~$\gerp$ regular for~$U$. 
Then each of  the analytic functions~$g_\ell$ has at most finitely many zeros on~$\Z_p$, hence on~$\Z$. 
\end{proof}

Actually, the Skolem-Mahler-Lech Theorem, as stated in Theorem~\ref{thsml}, applies to LRS over an arbitrary field of characteristic~$0$. To extend it to this generality, one more ingredient is needed, the \textit{Lech-Cassels Specialization Theorem}, see~\cite{Ca76}.

\subsection{Proof of Theorems~\ref{theasy},~\ref{thtwint} and~\ref{thtwrat}}
\label{ssproofs}

Let~$U$ be an LRS taking values in a number field~$\K$, and let~$\gerp$ be a prime of~$\K$ regular for~$U$, see Subsection~\ref{ssfunction}. We denote by~$p$ the rational prime below~$\gerp$.  In this section we prove the following two theorems, which generalize Theorems~\ref{theasy},~\ref{thtwint} and~\ref{thtwrat} from the Introduction.

\begin{theorem}
\label{theasyk}
Let~$a$ be a zero of~$U$.  Then there exist a positive integer~$Q$, a positive integer~$\kappa$ and an integer~$\tau$  such that 
${\nu_\gerp(U(n))= \kappa \nu_p(n-a)+\tau}$
for ${n\equiv a\pmod Q}$. 
\end{theorem}

\begin{theorem} \label{thtwratk}
Let~${a\in \Q}$ be such that
 there exists a  sequence of integers  $(n_k)$ satisfying
$$
\nu_\gerp(U(n_k))\to +\infty, \qquad \nu_p(n_k-a) \to +\infty. 
$$
Then~$a$ is a TRZ of~$U$. 
\end{theorem}

Theorem~\ref{theasyk} is more general than Theorem~\ref{theasy}, while Theorem~\ref{thtwratk} is more general than Theorem~\ref{thtwrat}  (and, \textit{a fortiori}, than Theorem~\ref{thtwint}). 

\begin{proof}[Proof of Theorem~\ref{theasyk}]
Let the integer~$N$ and the functions $g_0, \ldots, g_{N-1}$ be as in Proposition~\ref{prsko}. Let ${\ell\in \{0, \ldots, N-1\}}$ be such that ${a\equiv \ell \pmod N}$. Write ${a=\ell+bN}$ with ${b\in \Z}$. In the sequel, we denote ${g:=g_\ell}$. We have ${g(b)=0}$. 

Theorem~\ref{thfunml} implies that there exist positive integers~$k$ and~$\tau'$, and an integer~$\kappa'$ such that for ${z\equiv b\pmod{p^k}}$ we have 
\begin{equation}
\label{enupgz}
\nu_p(g(z))= \kappa'\nu_p(z-b)+e^{-1}\tau' . 
\end{equation}
Now set 
$$
Q:=Np^k, \qquad \kappa:=e\kappa', \qquad \tau:=\tau'-\kappa\nu_p(N). 
$$ 
Let ${n\equiv a\pmod Q}$. Then ${n\equiv \ell \pmod N}$, and for 
${m:=(n-\ell)/N}$ we have  
$$
m\equiv b \pmod {p^k}, \qquad \nu_p(m-b)=\nu_p(n-a)-\nu_p(N), \qquad g(m)=U(n). 
$$
Applying~\eqref{enupgz} with ${z=m}$, we obtain
$$
\nu_\gerp(U(n))=e\nu_p(g(m))= e\kappa'\nu_p(m-b)+\tau'=\kappa\nu_p(n-a)+\tau. 
$$
The theorem is proved. 
\end{proof}

\begin{proof}[Proof of Theorem~\ref{thtwratk}]
Once again, let~$N$ and  $g_0, \ldots, g_{N-1}$ be as in Proposition~\ref{prsko}. Let ${\ell\in \{0, \ldots, N-1\}}$ be such that ${n_k\equiv \ell\pmod N}$ holds for infinitely many~$k$. By taking a subsequence, we may assume that this holds for all~$k$. We denote ${g:=g_\ell}$.

Set ${b:=(a-\ell)/N}$ and ${m_k:=(n_k-\ell)/N}$. Then, ${m_k\to b}$ and ${g(m_k)\to 0}$ in the $p$-adic topology. Hence ${g(b)=0}$. 

(Note that, unlike in the proof of Theorem~\ref{theasyk}, we do not, in general, have ${g(b)=U(\ell+Nb)=U(a)}$; this would only be true  if ${b\in \Z}$. But this is not true in general:~$b$ is merely a rational number, not necessarily an integer.) 

Recall that 
$$
g(z)=g_\ell(z)= \sum_{i=1}^s\lambda_i^\ell h_i(z), \quad \text{where}\quad  h_i(z):=f_i(\ell+Nz) \exp(z\log(\lambda_i^N)). 
$$
Let~$A$ be the denominator of the rational number~$a$. Then 
$$
(h_i(b))^{AN}= \bigl(\lambda_i^\ell f_i(a)\bigr)^{AN} \exp(ANb\log(\lambda_i^N)). 
$$
Since ${ANb\in \Z}$, we have
${\exp(ANb\log(\lambda_i^N))=\lambda_i^{AN^2b}}$. Hence 
$$
(h_i(b))^{AN}=\lambda_i^{AN\ell+AN^2b}f_i(a)^{AN}=\bigl(\lambda_i^{a}f_i(a)\bigr)^{AN},  
$$
where we pick some definition for the rational power $\lambda_i^a$. Thus, 
$$
h_i(b)=\xi_i\lambda_i^af(a), 
$$
where~$\xi_i$ is a  root of unity.

We have proved that 
$$
0=g(b)= \sum_{i=1}^s\xi_i\lambda_i^af(a), 
$$
which exactly means that~$a$ is a TRZ of~$U$. The theorem is proved. 
\end{proof}



\section{Finiteness of twisted rational zeros}
\label{sfin}
In this section we prove Theorem~\ref{thfin}. Throughout this section, unless the contrary is stated explicitly,~$\K$ is a   field of characteristic~$0$. We fix an algebraic closure~$\overline{\K}$.  

For a positive integer~$n$ we fix ${\zeta_n\in \overline{\K}}$, a primitive $n^\tho$ root of unity. Recall that we denote by~$\mu_n$ the group of $n^\tho$ roots of unity, and by~$\mu_\K$ the group of roots of unity in~$\K$.  We denote by $\K^n$ the set of $n^\tho$ powers in~$\K$:
$$
\K^n:=\{\alpha^n:\alpha \in \K\}. 
$$

We denote by $\K_\ab$ the maximal abelian subfield of~$\K$;  that is, the maximal subfield of~$\K$ which is an abelian extension of~$\Q$. 



\subsection{Powers in fields}
\label{sspows}

The following result is due to Chevalley~\cite{Ch51} and Bass~\cite{Ba65}. The proofs can be also found in~\cite{Sm70} and~\cite{Bi23}. 
\begin{theorem}
\label{thbass}
[Chevalley, Bass]
Let~$\K$ be a finitely generated field of characteristic~$0$ (in particular,~$\K_\ab$ is a finite extension of~$\Q$). 
Then there exists a positive integer~$\Lambda$, depending only on the degree ${d:=[\K_\ab:\Q]}$, such that for every positive integer~$n$ the following holds: if ${\alpha\in \K}$ is a $\Lambda n^\tho$ power in  $\K(\zeta_{\Lambda n})$, then~$\alpha$ is an $n^\tho$  power in~$\K$.   In symbols: 
\begin{equation}
\label{echba}
\K(\zeta_{\Lambda n})^{\Lambda n}\cap \K\subset \K^n \qquad (n=1,2,3,\ldots). 
\end{equation} 
\end{theorem}

The smallest positive integer~$\Lambda$ satisfying~\eqref{echba} will be 
called the
\textit{Chevalley-Bass number} of the field~$\K$; see \cite[Section~6]{Bi23}. 
 
It might not be easy to determine the Chevalley-Bass number of a given field~$\K$, but it is easy to estimate it in terms of  ${d:=[\K_\ab:\Q]}$. For instance, it is shown in \cite[Section~6.1]{Bi23} that, when ${d\ge 3}$, the Chevalley-Bass number~$\Lambda$ satisfies
${\Lambda \le \exp(d^{2/\log\log d})}$.

It will be convenient to introduce the following notion. 
For ${\alpha \in \K^\times}$ we define the \textit{Kummer exponent} of~$\alpha$ in~$\K$ as the biggest positive integer~$n$ such that ${\alpha\xi \in \K^n}$ for some root of unity ${\xi \in \K}$. In symbols:
$$
\varrho_\K(\alpha):=\max\{n: \alpha \in \mu_\K \K^n\}
$$
(recall that~$\mu_\K$ denotes the group of roots of unity in~$\K$).  
Clearly, ${\varrho_\K(\alpha)=\infty}$ if~$\alpha$ is  a root of unity, and,  when~$\K$ is a finitely generated field, $\varrho_\K(\alpha)$ is finite if~$\alpha$ is not a root of unity. 

\begin{proposition}
\label{proot}
Let ${\alpha\in \K}$ be such that $\varrho_\K(\alpha)$ is finite, and~$n$ a positive integer.   

\begin{enumerate}
\item
\label{ieasy}
We have ${\alpha \in \mu_\K\K^n}$ if and only if ${n\mid \varrho_\K(\alpha)}$. 

\item
\label{iharder}
Let ${\alpha^{1/n}\in \overline{\K}}$ be some determination of the $n^\tho$ root, and  ${\xi\in \overline{\K}}$ a root of unity. Then the degree ${[\K(\alpha^{1/n}\xi):\K]}$ is a multiple of ${n/\gcd(\varrho_\K(\alpha), n)}$. 
\end{enumerate}
\end{proposition}

\begin{proof}
Item~\ref{ieasy} follows immediately from the definition. To prove item~\ref{iharder}, define 
$$
\L:=\K(\alpha^{1/n}\xi), \qquad m:=[\L:\K], \qquad\rho:=\varrho_\K(\alpha), \qquad d:=\gcd(m,n). 
$$
All conjugates of $\alpha^{1/n}$ over~$\K$ are equal to $\alpha^{1/n}$ times a root of unity. Hence ${\beta:=\NN_{\L/\K}(\alpha^{1/n}\xi)}$ is $\alpha^{m/n}$ times a root of unity. Let ${r,s\in \Z}$ be such that ${mr+ns=d}$. Then ${\gamma:=\alpha^s\beta^r}$ is $\alpha^{d/n}$ times a root of unity. Since $\gamma^{n/d}$ is~$\alpha$ times a root of unity, we have ${n/d\mid \rho}$. Hence ${n/d\mid \gcd( \rho,n)}$. It follows that ${n/\gcd(\rho, n)}$ divides~$d$. Hence it divides~$m$. 
\end{proof}


\subsection{Equations in roots of unity}
\label{ssrun}

The following result is due to Dvornicich and Zannier \cite{DZ00,Za95}, who improved on the previous work of Mann~\cite{Ma65} and of Conway and Jones~\cite{CJ76}. In this subsection~$\K$ is a finitely generated field of characteristic~$0$. 

\begin{theorem}
\label{thdz}
[Dvornicich, Zannier]
Let ${\alpha_1, \ldots, \alpha_s}$ be non-zero elements of~$\K$, and ${\xi_1, \ldots, \xi_s\in \overline{\K}}$ roots of unity. Assume that 
\begin{equation*}
\alpha_1\xi_1+ \cdots +\alpha_s\xi_s=1, 
\end{equation*}
and no proper sub-sum of the sum on the left vanishes: ${\sum_{i\in I}\alpha_i\xi_i\ne 0}$ when ${\varnothing \subsetneq I\subsetneq\{1, \ldots, s\}}$. 
Then the order of the multiplicative group generated by 
${\xi_1, \ldots, \xi_s}$ is effectively bounded in terms of  ${d=[\K_\ab:\Q]}$ and~$s$. 
\end{theorem}

In fact, 
Dvornicich and Zannier prove that, denoting by~$r$ the order of the group generated by 
${\xi_1, \ldots, \xi_s}$, we have the following properties:
\begin{itemize}
\item
if ${p^{a+1}\mid r}$ for some positive integer~$a$ then 
${p^a\mid 2d}$;

\item
${\displaystyle 
s+1\ge\dim_{\K}(\K+\K\xi_1+\cdots+\K\xi_s)\ge 1+\sum_{p\|r}\left(\frac{p-1}{\gcd(d,p-1)}-1\right) }$.  
\end{itemize}
Clearly, using these properties, it is easy to bound~$r$ explicitly in terms of~$d$ and~$s$. 

\subsection{A Kummer property}

As before,~$\K$ is a finitely generated field of characteristic~$0$. Recall that we denote by $\K_\ab$ the maximal abelian subfield of~$\K$. Let~$\Gamma$ be the division group of the multiplicative group~$\K^\times$:
$$
\Gamma:=\{a\in \overline{\K}^\times: \text{$a^n\in \K^\times$ for some positive integer~$n$}\}. 
$$
The following key proposition is, essentially, due to Laurent~\cite{La84}. 

\begin{proposition}
\label{prkumlaur}
Let ${\alpha_1, \ldots, \alpha_s \in \Gamma} $ be such that ${\alpha_1+\cdots+\alpha_s=1}$, and no proper sub-sum of the sum ${\alpha_1+\cdots+\alpha_s}$ vanishes.  Let~$\Lambda$ 
be the Chevalley-Bass number of~$\K$ (see Section~\ref{sspows}). Then there exist roots of unity ${\xi_1, \ldots, \xi_s\in \overline{\K}}$ such that 
\begin{equation}
\alpha_i^\Lambda \xi_i\in \K \qquad (i=1, \ldots, m),
\end{equation}
\end{proposition}

\begin{proof}
We follow  Laurent \cite[Section~2.2]{La84}, but we replace the cohomological argument by a reference to Theorem~\ref{thbass}. 

Let~$n$ be a positive integer such that ${\alpha_1^n, \ldots, \alpha_s^n\in \K}$. Denote by~$G$ the  Galois group ${\Gal(\overline{\K}/\K(\zeta_n))}$.  For ${i\in\{1, \ldots, s\}}$ let ${\chi_i:G\to  \K(\zeta_n)}$ be the character of~$G$ defined by ${\sigma\mapsto\sigma(\alpha_i)/\alpha_i}$. 
Since ${\sigma(\alpha_1)+\cdots +\sigma(\alpha_s)=1}$ for every ${\sigma\in G}$, we have 
${\alpha_1\chi_1+\cdots +\alpha_s\chi_s=1}$.

We claim that the characters ${\chi_1, \ldots,\chi_s}$ are all trivial: 
\begin{equation}
\label{ealltriv}
\chi_1=\cdots=\chi_s=1.
\end{equation}
Indeed, defining ${\alpha_0:=-1}$ and ${\chi_0:=1}$, we have ${\alpha_0\chi_0+\cdots +\alpha_s\chi_s=0}$. After renumbering the characters ${\chi_1, \ldots,\chi_s}$, we may assume that for some ${r\ge 0}$ the characters ${\chi_0, \ldots, \chi_r}$ are distinct, and each of the remaining ${\chi_{r+1}, \ldots, \chi_s}$ is equal to one of ${\chi_0, \ldots, \chi_r}$. 
For ${k=0, \ldots, r}$ define ${I_k:=\{i: \chi_i=\chi_k\}}$. Then 
$$
\sum_{k=0}^r\chi_k\sum_{i\in I_k}\alpha_i=0. 
$$
Artin's Theorem on Characters (see, for instance, Theorem~4.1 in \cite[Chapter~VI]{La02}) implies that ${\sum_{i\in I_k}\alpha_i=0}$ for every~$k$. Since no proper sub-sum  of ${\alpha_1+\cdots+\alpha_s}$ vanishes, this is possible only if ${r=0}$, which proves~\eqref{ealltriv}.

It follows from~\eqref{ealltriv} that ${\alpha_1, \ldots, \alpha_s\in \K(\zeta_n)}$. Hence each $\alpha_i^{\Lambda n}$ is an $\Lambda n^\tho$ power in $\K(\zeta_n)$. 
Theorem~\ref{thbass} implies that $\alpha_i^{\Lambda n}$ is an $n^\tho$ power in~$\K$. It follows that ${\alpha_i^\Lambda\xi_i\in \K}$ for some root of unity~$\xi_i$, as wanted. 
\end{proof}

\subsection{Proof of Theorem~\ref{thfin}}
\label{ssprfin}

Let~$U$ be a non-degenerate LRS with values in a field of characteristic~$0$ and with  Binet expansion 
\begin{equation}
\label{ebin}
U(n)=f_1(n)\lambda_1^n+\cdots+f_s(n) \lambda_s^n. 
\end{equation}
Let~$\K$ be a finitely generated field, containing ${\lambda_1,\ldots, \lambda_s}$ and the coefficient of the polynomials $f_1, \ldots, f_s$.  Recall that we denote by~$\K_\ab$ the maximal abelian subfield of~$\K$. Since~$\K$ is finitely generated,~$\K_\ab$ is a finite extension of~$\Q$, and we denote by~$d$ its degree over~$\Q$.  

Recall that we denote by $\varrho_\K(\alpha)$ the Kummer exponent of ${\alpha \in \K}$, see Subsection~\ref{sspows}. Since the~$U$ is non-degenerate and the field~$\K$ is finitely generated, we have ${\varrho_\K(\lambda_i/\lambda_j)<\infty}$ when ${i\ne j}$. 
We set
\begin{equation}
\label{edefrho}
\rho:=\gcd\{\varrho_\K(\lambda_i/\lambda_j): 1\le i<j\le s\}. 
\end{equation}

Let~$a$ be a TRZ of~$U$. Recall that this means the following: there exist roots of unity ${\xi_1, \ldots, \xi_s\in \overline{\K}}$ such that for some determinations of ${\lambda_1^a, \ldots, \lambda_s^a\in \overline{\K}}$ we have 
\begin{equation}
\label{etrz}
\xi_1f_1(a)\lambda_1^a+\cdots + \xi_sf_s(a)\lambda_s^a=0. 
\end{equation}
We call the TRZ~$a$ \textit{primitive} if no proper sub-sum of the sum in~\eqref{etrz} vanishes; that is, if ${\varnothing \subsetneq I\subsetneq \{1,\ldots,s\}}$ then ${\sum_{i\in I}\xi_if_i(a)\lambda_i^a\ne 0}$.

\begin{proposition}
\label{pmar}
Let~$\K$ be as above (that is, a finitely generated field, containing ${\lambda_1,\ldots, \lambda_s}$ and the coefficient of the polynomials $f_1, \ldots, f_s$), and~$\Lambda$ the Chevalley-Bass number of~$\K$. 
Let~$a$ be a primitive TRZ of~$U$, and ${\xi_1, \ldots, \xi_s}$ roots of unity satisfying~\eqref{etrz}. Assume that ${s\ge 2}$.  Then the denominator of the rational number~$a$ divides~$\Lambda\rho$, where~$\rho$ is  defined in~\eqref{edefrho}; in particular, the denominator is bounded effectively in terms of ${d:=[\K_\ab:\Q]}$ and~$\rho$.  Moreover,  the orders of the roots of unity ${\xi_i/\xi_j}$ are effectively bounded in terms of~$d$,~$\rho$ and~$s$. 
\end{proposition}

\begin{proof}
Since ${s\ge 2}$ and no proper sub-sum of the sum in~\eqref{etrz} vanishes, we have ${f_i(a)\ne 0}$ for ${i=1, \ldots, s}$. There will be no loss of generality to assume that ${\xi_s=1}$; so, instead of~\eqref{etrz} we have
\begin{equation}
\label{etrzone}
\xi_1f_1(a)\lambda_1^a+\cdots + \xi_{s-1}f_{s-1}(a)\lambda_{s-1}^a+ f_s(a)\lambda_s^a=0.  
\end{equation} 
Applying Proposition~\ref{prkumlaur} to the relation 
\begin{equation}
\label{erelat}
\sum_{i=1}^{s-1}-\xi_i\frac{f_i(a)}{f_s(a)}\left(\frac{\lambda_i}{\lambda_s}\right)^a=1, 
\end{equation}
we obtain the following: there exist  roots of unity ${\eta_1, \ldots, \eta_{s-1}}$ such that 
\begin{equation*}
\left({\lambda_i}/{\lambda_s}\right)^{\Lambda a}\eta_i\in \K \qquad (i=1, \ldots, s-1). 
\end{equation*}
Write  ${a=k/\ell}$, where~$k$ and~$\ell$ are co-prime integers. We want to show that ${\ell\mid \Lambda\rho}$. We have 
${\varrho_\K\bigl((\lambda_i/\lambda_s)^{\Lambda k}\bigr) = \Lambda k\varrho_\K(\lambda_i/\lambda_s)}$.
Proposition~\ref{proot} implies that the quotient 
${\ell/\gcd\bigl(\ell, \Lambda k\varrho_\K(\lambda_i/\lambda_s)\bigr)}$ divides the degree ${[\K((\lambda_i/\lambda_s)^{\Lambda a}\eta_i):\K]}$. But this degree is~$1$, which implies that ${\ell\mid \Lambda k\varrho_\K(\lambda_i/\lambda_s)}$. Since~$\ell$ and~$k$ are coprime, this implies that 
${\ell\mid \Lambda \varrho_\K(\lambda_i/\lambda_s)}$
for ${i=1, \ldots, s-1}$.

We have clearly
${\rho=\gcd\{\varrho_\K(\lambda_i/\lambda_s): i=1, \ldots, s-1\}}$. 
It follows that ${\ell\mid \Lambda \rho}$, which proves  the first statement of the proposition.

Now let us bound the orders of the roots of unity~$\xi_i$. 
Let~$\L$ be the field, generated over~$\K$ by ${(\lambda_1/\lambda_s)^a, \ldots,  (\lambda_{s-1}/\lambda_s)^a}$. Since the denominator of~$a$ divides $\Lambda \rho$, we have 
$$
[\L_\ab:\K_\ab]\le [\L:\K]\le (\Lambda \rho)^{s-1}.
$$
Hence ${[\L_\ab:\Q]\le d(\Lambda \rho)^{s-1}}$; in particular,  ${[\L_\ab:\Q]}$ is effectively bounded in terms of $d$,~$s$ and~$\rho$.

Applying Theorem~\ref{thdz} to relation~\eqref{erelat}, we bound the orders of the roots of unity $\xi_i$ in terms of ${[\L_\ab:\Q]}$ and~$s$. Hence it is bounded in terms of $d$,~$s$ and~$\rho$, as wanted. The proposition is proved.  
\end{proof}

Combining this proposition with the Skolem-Mahler-Lech Theorem, we obtain the following consequence.

\begin{corollary}
\label{cotrz}
Let~$U$ be a non-degenerate LRS over a field of characteristic~$0$. Then~$U$ admits at most finitely many primitive TRZs. More precisely, if~\eqref{ebin} is the Binet expansion of~$U$, then there exist at most finitely many $s$-tuples  $(a,\xi_1, \ldots, \xi_{s-1})$ such that~$a$ is a rational number, ${\xi_1, \ldots, \xi_{s-1}}$ are roots of unity, and~\eqref{etrzone} holds. 
\end{corollary}

\begin{proof}
If ${s=1}$ then ${f_1(a)=0}$, which is possible only for finitely many~$a$.

From now on we assume that ${s\ge 2}$. Proposition~\ref{pmar} implies that ${a=n/\Lambda \rho}$, where ${n\in \Z}$, and that there are at most finitely many choices for $(\xi_1, \ldots, \xi_{s-1})$ in~\eqref{etrzone}. Pick some determinations ${\theta_i:=\lambda_i^{1/\Lambda \rho}}$, so that ${\lambda_i^a=\theta_i^n\eta_i}$, where $\eta_i$ are $\Lambda \rho^\tho$ roots of unity. Then
\begin{equation}
\label{eetag}
\xi_1\eta_1g_1(n)\theta_1^n+\cdots + \xi_{s-1}\eta_{s-1}g_{s-1}(n)\theta_{s-1}^n +\eta_sg_s(n)\theta_s^n=0,  
\end{equation}
where ${g_i(t):=f_i(t/\Lambda \rho)}$. 

The left-hand side of~\eqref{eetag} is a non-degenerate LRS, and the Skolem-Mahler-Lech Theorem implies that there can be at most finitely many~$n$ for every fixed choice of the roots of unity~$\xi_i$ and~$\eta_i$. Since there are at most finitely many choices for $(\xi_1, \ldots, \xi_{s-1})$ and for $(\eta_1, \ldots, \eta_{s})$, the result follows. 
\end{proof}

Now we are ready to complete the proof of Theorem~\ref{thfin}. Let~$a$ be a TRZ of~$U$, so that~\eqref{etrz} holds for some choice of roots of unity~$\xi_i$. 
Let~$I$ be a minimal non-empty subset of ${\{1, \ldots, s\}}$ such that ${\sum_{i\in I}\xi_if_i(a)\lambda_i^a= 0}$. Then~$a$ is a primitive TRZ of the LRS~$U_I$, defined by  
\begin{equation}
\label{eui}
U_I(n):=\sum_{i\in I}f_i(n)\lambda_i^n. 
\end{equation}
Corollary~\ref{cotrz} tells us that~$U_I$ may have at most finitely many primitive TRZs. Since there are finitely many possible~$I$, Theorem~\ref{thfin} is proved.

\subsection{An explicit result for $\Q$-valued linear recurrence sequences}
\label{ssq}

Let~$U$ be an LRS over a field of characteristic~$0$ with Binet expansion~\eqref{ebin}, and let~$a$ be a rational number. If~$a$ is a common root of the polynomials ${f_1, \ldots, f_s}$, then it is, clearly, a rational zero of~$U$.  Such rational zeros will be called \textit{trivial}; there are only finitely many of them, and in many interesting cases (for instance, if at least one of ${f_1, \ldots, f_s}$ is constant) there are none.

Arguing as at the end of Subsection~\ref{ssprfin}, we obtain the following: the denominator of a non-trivial TRZ~$a$ of~$U$ is bounded in terms of~$d$,~$\rho$ and~$s$. Indeed, if ${f_i(a)\ne 0}$ for some~$i$, then there exists a set ${I\subset \{1,\ldots, s\}}$ having at least~$2$ elements such that the LRS~$U_I$, defined in~\eqref{eui}, has~$a$ as a primitive TRZ.  Now Proposition~\ref{pmar} implies that the denominator of~$a$ is bounded in terms of~$d$,~$\rho$ and~$s$.

Note that ${s\le m}$, where~$m$ denoted the order of the LRS~$U$. Hence the denominator of~$a$ is bounded in terms of~$d$,~$\rho$ and~$m$.

In the most interesting special case when~$U$ is a $\Q$-valued LRS of order~$m$, we can take~$\K$ in Proposition~\ref{pmar} as the splitting field of the characteristic polynomial of~$U$. With this choice of~$\K$, the degree~$d$ is bounded in terms of~$m$. Hence the denominator of~$a$ is bounded in terms of~$m$ and~$\rho$. We are going to make it totally explicit.

\begin{proposition}
Let~$U$ be a $\Q$-valued LRS of order~$m$ with Binet expansion~\eqref{ebin},~$\K$  the splitting field of the characteristic polynomial of~$U$, and~$\rho$ as in~\eqref{edefrho}. Let~$a$ be a non-trivial TRZ of~$U$. Then the denominator of~$a$ does not exceed ${\rho\exp\exp(m/\log m)}$. 
\end{proposition}

\begin{proof}
As we have just seen, if~$a$ is a non-trivial TRZ of~$U$, then there exists ${I\subset \{1, \ldots, s\}}$ with ${\#I\ge 2}$ such that~$a$ is a primitive TRZ of~$U_I$. Proposition~\ref{pmar} implies that the denominator of~$a$ is bounded by $\rho\Lambda$. And we have ${\Lambda \le \exp\exp(m/\log m)}$, see \cite[Proposition~6.5]{Bi23}. This completes the proof. 
\end{proof}

\section{Twisted rational zeros of the Tribonacci sequence}
\label{stribo}
In this section we prove Theorem~\ref{thtribo}. As in the proof of Theorem~\ref{thfin}, the principal part is bounding the denominators of the TRZ, see Proposition~\ref{prdenom}. Instead of adapting the general argument of Section~\ref{sfin}, using the Theorems of Chevalley-Bass and of Dvornicich-Zannier, we  use an elementary ad hoc argument.

In this section we  denote by ${\overline{\Q}\subset\C}$  the field of all complex algebraic numbers, and by ${x\mapsto\overline{x}}$ the complex conjugation. 

\subsection{The roots}

Let ${\lambda_1,\lambda_2, \lambda_3\in \overline{\Q}}$ be the complex roots of the characteristic polynomial 
$$
{P(X):=X^3-X^2-X-1}.
$$
Then
$$
T(n) =\alpha_1\lambda_1^n+\alpha_2\lambda_2^n+\alpha_3\lambda_3^n, \qquad \alpha_i:= P'(\lambda_i)^{-1}\lambda_i. 
$$
One of the roots ${\lambda_1, \lambda_2, \lambda_3}$  is real and the other two are complex conjugate. We will assume that
${\lambda_1\in \R}$ and ${\lambda_3=\overline{\lambda_2}}$.

We denote ${\K_i:=\Q(\lambda_i)}$ and  ${\L:=\Q(\lambda_1,\lambda_2)}$, the splitting field of~$P$. We have 
${[\K_i:\Q]=3}$ and ${[\L:\Q]=6}$. The 
maximal abelian subfield $\L_\ab$ is $\Q(\sqrt{-11})$, because the discriminant of~$P$ is $-44$. The $6$~numbers 
\begin{equation}
\label{equotients}
\lambda_i/\lambda_j \qquad (1\le i\ne j\le 3)
\end{equation}
form a full Galois orbit over~$\Q$; in particular, for ${i\ne j}$ we have ${\L=\Q(\lambda_i/\lambda_j)}$.

In the following proposition we collect some  less obvious properties of the roots $\lambda_i$, to be used later. 
\begin{proposition}
\label{prlami}
\begin{enumerate}

\item
\label{inotunits}
For ${i\ne j}$, the quotient ${P'(\lambda_i)/P'(\lambda_j)}$ is not a Dirichlet unit. 

\item
\label{ikumexp}
Let~$\lambda$ be one of ${\lambda_1,\lambda_2,\lambda_3}$ and ${\K:=\Q(\lambda)}$. 
Then~$\lambda$ is not an $m^\tho$ power in~$\K$ for any integer ${m>1}$, and neither is ${-\lambda}$. In terms of the Kummer exponent, defined in Subsection~\ref{sspows}, this can be stated as  ${\varrho_{\K}(\lambda)=1}$. 

\item
\label{icubes}
The quotients $\lambda_i/\lambda_j$ are cubes in~$\L$. More precisely, for ${1\le i,j\le 3}$ there exists a unique ${\theta_{ij}\in \L}$ such that ${\lambda_i/\lambda_j=\theta_{ij}^3}$.

\end{enumerate}
\end{proposition}

\begin{proof}
If, say, ${P'(\lambda_1)/P'(\lambda_2)}$ is a unit, then so is every ${P'(\lambda_i)/P'(\lambda_j)}$. It follows that the $3$~algebraic integers $P'(\lambda_i)$ generate the same  principal ideal in~$\OO_\L$. This ideal must divide the sum
$
\sum_{i=1}^3 P'(\lambda_i)=4 
$,
which implies that the product 
$
\prod_{i=1}^3 P'(\lambda_i) 
$
must be a power of~$2$. But 
\begin{equation}
\label{eff}
\prod_{i=1}^3 P'(\lambda_i)=44, 
\end{equation}
a contradiction. This proves item~\ref{inotunits}.

In the proof of item~\ref{ikumexp} we will use the notion of the \textit{absolute logarithmic height} $\height(\cdot)$ of an algebraic number. The definition can be found in many sources, say, in \cite[Section~1.5.7]{BG06}. We will need only the following properties:  if~$\gamma$ is an algebraic integer of degree~$d$ with conjugates ${\gamma_1, \ldots, \gamma_d \in \C}$, and~$m$ is a positive integer, then 
$
d\height(\gamma) =\sum_{i=1}^d\max\{\log|\gamma_i|,0\} 
$,
and ${\height(\gamma^m)=m\height(\gamma)}$. In particular, 
$$
3\height(\lambda) =\log\lambda_1 <0.61, 
$$
because 
${\lambda_1>1}$ and ${|\lambda_2|=|\lambda_3|<1}$.

Now assume that ${\lambda=\gamma^m}$ for some ${\gamma \in \K}$ and ${m>1}$. Then ${\height(\lambda)=m\height(\gamma)}$. On the other hand, the famous result of Smyth~\cite{Sm71} implies that 
$$
3\height(\gamma)\ge \log \vartheta >0.28. 
$$ 
where~$\vartheta$ is the real root of the polynomial ${X^3-X-1}$. Hence 
$$
m\le \frac{\log\lambda}{\log\vartheta} < 2.2. 
$$
It follows that ${m=2}$. Hence~$\gamma$ is a root of the polynomial ${P(X^2)}$. However, this polynomial is irreducible over~$\Q$, which means that~$\gamma$ is of degree~$6$, a contradiction. 
In a similar fashion one shows that $-\lambda$ is not a proper power in~$\K$.  This proves item~\ref{ikumexp}.

In item~\ref{icubes} uniqueness is clear, because ${\zeta_3\notin \L}$, so we only have to prove existence. 
Using \textsf{PARI}~\cite{pari} (or another similar tool), we calculate the $X$-resultant of the polynomials $P(X)$ and $P(XY)$. It is a polynomial in~$Y$ of degree~$9$, whose roots are exactly the $9$~quotients $\lambda_i/\lambda_j$. It has a root~$1$ of multiplicity~$3$, which corresponds to the $3$~quotients $\lambda_i/\lambda_i$, and it factors as ${(Y-1)^3R(Y)}$, where 
$$
R(Y):=Y^6 + 4Y^5 + 11Y^4 + 12Y^3 + 11Y^2 + 4Y + 1
$$
is the irreducible polynomial whose roots are the quotients~\eqref{equotients}.

Polynomial $R(Y^3)$ is reducible over~$\Q$: it has an irreducible factor of degree~$6$
$$
Q(Y):= Y^6 + Y^5 + 2Y^4 + 3Y^3 + 2Y^2 + Y + 1
$$
and another irreducible factor of degree~$12$. Let~$\theta$ be a root of~$Q$.  Then the field $\Q(\theta)$ is of degree~$6$, and $\theta^3$ is one of the quotients~\eqref{equotients}; in particular, $\Q(\theta)$ contains~$\L$. By equality of degrees we obtain ${\Q(\theta)=\L}$, which completes the proof of item~\ref{icubes}. 
\end{proof}


\subsection{The denominator}
In this subsection we prove that the denominator of a TRZ divides~$3$.

\begin{proposition}
\label{prdenom}
Let~$a$ be a TRZ of the Tribonacci LRS\@. Then ${3a\in \Z}$. 
\end{proposition}

We will use the following very simple lemma.

\begin{lemma}
\label{lstrange}
Let ${\gamma\in \C}$ be a complex  number such that ${\gamma/\overline{\gamma}}$ is not a root of unity, and ${\delta\in \R}$ a real number. Then the equation ${\gamma\eta+\overline{\gamma}\eta'=\delta}$ may have at most one solution in roots of unity $(\eta,\eta')$. This solution satisfies ${\eta\in \Q(\gamma,\overline{\gamma},\delta)}$ and ${\eta'=\overline{\eta}}$. 
\end{lemma}

\begin{proof}
If ${\eta, \eta'}$ are roots of unity such that ${\gamma\eta+\overline{\gamma}\eta'\in \R}$, then the complex numbers $\overline{\gamma}\eta'$ and $\overline{\gamma}\overline{\eta}$ have the same imaginary part: ${\overline{\gamma}\eta'-\gamma\overline{\eta'}= \overline{\gamma}\overline{\eta}-\gamma\eta}$. 
If ${\eta\ne \overline{\xi}}$ then 
$$
\frac{\gamma}{\overline{\gamma}}= \frac{\eta'-\overline{\eta}}{\overline{\eta'}-\eta}= -\frac{\eta'}{\eta}, 
$$
contradicting the hypothesis that ${\gamma/\overline{\gamma}}$ is not a root of unity. Hence ${\eta'=\overline{\eta}}$. 

Thus,~$\eta$ is a root of the polynomial 
$$
F(X):= X^2 -(\delta/\gamma)X+\overline{\gamma}/\gamma\in \Q(\gamma,\overline{\gamma},\delta)[X]. 
$$
Since the free term $\overline{\gamma}/\gamma$ is not a root of unity, the other root of~$F$ cannot be a root of unity. This proves that there may exist only one possible~$\eta$ for given~$\gamma$ and~$\delta$.

Furthermore, if~$F$ is irreducible over $\Q(\gamma,\overline{\gamma},\delta)$ then its other root is a root of unity as well, which is impossible, as we just saw. Hence~$F$ is reducible, which implies that ${\eta \in \Q(\gamma,\overline{\gamma},\delta)}$. The lemma is proved. 
\end{proof}

\begin{proof}[Proof of Proposition~\ref{prdenom}]
Let~$m$ be the denominator of~$a$; that is, ${a=n/m}$, where ${m,n\in \Z}$ are coprime and ${m>0}$. Let $\lambda_1^{1/m}$ be the positive real $m^\tho$ root,  $\lambda_2^{1/m}$ some complex $m^\tho$ root,  and we define $\lambda_3^{1/m}$ as the complex conjugate of $\lambda_2^{1/m}$. With this choice of $m^\tho$ roots we have 
\begin{equation}
\label{eprodone}
\lambda_1^{1/m}\lambda_2^{1/m}\lambda_3^{1/m}=1. 
\end{equation}
Once the $m^\tho$ roots are defined, the rational powers $\lambda_i^a$ are well-defined as $(\lambda_i^{1/m})^n$. Note that ${\lambda_3^a=\overline{\lambda_2^a}}$. 

Since~$a$ is a TRZ, we have ${\alpha_1\lambda_1^a\eta_1+\alpha_2\lambda_2^a\eta_2+\alpha_2\lambda_3^a\eta_3=0}$ for some roots of unity ${\eta_1,\eta_2,\eta_3}$. We may clearly assume that ${\eta_1=1}$, in which case Lemma~\ref{lstrange} implies that ${\eta_3=\overline{\eta_2}}$. In the sequel we write~$\eta_2$ as~$\eta$ and~$\eta_3$ as~$\overline{\eta}$; that is, we have 
\begin{equation}
\label{etrztrib}
\alpha_1\lambda_1^a+\alpha_2\lambda_2^a\eta+\alpha_3\lambda_3^a\overline{\eta}=0. 
\end{equation}

Recall that we denote ${\K_1=\Q(\lambda_1)}$. Fix an element~$\sigma$ in the absolute Galois group 
${G_1:=\Gal(\overline{\Q}/\K_1)}$. Then ${(\lambda_1^{1/m})^\sigma=\lambda_1^{1/m}\xi_1}$ for some ${\xi_1\in \mu_m}$. The group ${H:=\Gal(\overline{\Q}/\L)}$ is an index~$2$ subgroup of~$G_1$. If ${\sigma\in H}$ then ${\lambda_2^\sigma=\lambda_2}$ and ${\lambda_3^\sigma=\lambda_3}$, in which case there exist  ${\xi_2,\xi_3\in \mu_m}$ such that 
\begin{equation}
\label{enoswitch}
(\lambda_2^{1/m})^\sigma=\lambda_2^{1/m}\xi_2, \qquad (\lambda_3^{1/m})^\sigma=\lambda_3^{1/m}\xi_3. 
\end{equation}
If ${\sigma\notin H}$ then ${\lambda_2^\sigma=\lambda_3}$ and ${\lambda_3^\sigma=\lambda_2}$; in this case there exist  ${\xi_2,\xi_3\in \mu_m}$ such that 
\begin{equation}
\label{eswitch}
(\lambda_2^{1/m})^\sigma=\lambda_3^{1/m}\xi_3, \qquad (\lambda_3^{1/m})^\sigma=\lambda_2^{1/m}\xi_2. 
\end{equation}
Let us apply~$\sigma$ to equalities~\eqref{eprodone} and~\eqref{etrztrib}.  We obtain 
\begin{align}
\label{eprodroots}
&1=1^\sigma= (\lambda_1^{1/m}\lambda_2^{1/m}\lambda_3^{1/m})^\sigma= \lambda_1^{1/m}\lambda_2^{1/m}\lambda_3^{1/m}\xi_1\xi_2\xi_3= \xi_1\xi_2\xi_3,  \\
\label{etrztribsign}
&\alpha_1^\sigma\lambda_1^a\xi_1^n+\alpha_2\lambda_2^a\xi_2^n\eta'+\alpha_3\lambda_3^a\xi_3^n\eta''=0, \qquad
(\eta',\eta'')=
\begin{cases}
(\eta^\sigma,\overline{\eta}^\sigma),& \sigma \in H, \\
(\overline{\eta}^\sigma,\eta^\sigma),& \sigma \notin H. 
\end{cases} 
\end{align}
Note that ${\alpha_1^\sigma=\alpha_1}$ because ${\alpha_1\in \K_1}$,  and that ${\eta'\eta''=(\eta\overline{\eta})^\sigma=1}$.

Rewrite~\eqref{etrztribsign} as 
$$
\alpha_1\lambda_1^a+\alpha_2\lambda_2^a\left(\frac{\xi_2}{\xi_1}\right)^n\eta'+\alpha_3\lambda_3^a\left(\frac{\xi_3}{\xi_1}\right)^n\eta''=0.  
$$
Comparing this to~\eqref{etrztrib}, the uniqueness statement in Lemma~\ref{lstrange} implies that  
$$
\left(\frac{\xi_2}{\xi_1}\right)^n\eta'= \eta, \qquad \left(\frac{\xi_3}{\xi_1}\right)^n\eta''= \overline{\eta}. 
$$
Multiplying these equalities, we obtain ${(\xi_2\xi_3/\xi_1^2)^n=1}$. Since ${\xi_1,\xi_2,\xi_3 \in \mu_m}$ and $m,n$ are coprime, this implies that ${\xi_2\xi_3/\xi_1^2=1}$, which, together with~\eqref{eprodroots}, implies that ${\xi_1^3=1}$. 

We have proved the following: for any ${\sigma \in G_1}$ there exist ${\xi_1\in \mu_3}$ such that ${(\lambda_1^{1/m})^\sigma= \lambda_1^{1/m}\xi_1}$. If ${\xi_1=1}$ for every ${\sigma\in G_1}$ then ${\lambda_1^{1/m}\in \K_1}$. Now assume that ${\xi_1=\zeta_3}$ for some~$\sigma$. Since~$\K_1$ is a real field, any Galois orbit over~$\K_1$ must be stable under the complex conjugation. Hence the Galois orbit of $\lambda_1^{1/m}$ over~$\K_1$ is ${\lambda_1^{1/m},\lambda_1^{1/m}\zeta_3, \lambda_1^{1/m}\overline{\zeta_3}}$.

Thus,  ${[\K_1(\lambda_1^{1/m}):\K_1]\in \{1,3\}}$. On the other hand, item~\ref{ikumexp} of Proposition~\ref{prlami} implies that~$\lambda_1$ is not a $p^\tho$ power in~$\K_1$ for any $p$, and ${-\lambda_1}$ is not a square in~$\K_1$. Hence ${[\K_1(\lambda_1^{1/m}):\K_1]=m}$ by Lemma~\ref{llang}. It follows that ${m\in \{1,3\}}$, as wanted. 
\end{proof}

We also need to take care of the roots of unity occurring in the definition of TRZ. 

\begin{proposition}
\label{prnoroots}
Let~$a$ be a TRZ of the Tribonacci LRS\@. Then, 
with suitable definitions of the rational powers $\lambda_i^a$ we have 
\begin{align}
\label{esumm}
\alpha_1\lambda_1^a+ \alpha_2\lambda_2^a+\alpha_3\lambda_3^a&=0, \\
\label{eprodd}
\lambda_1^a \lambda_2^a\lambda_3^a&=1
\end{align} 
\end{proposition}

\begin{proof}
We have ${a=n/3}$, with ${n\in \Z}$. We define $\lambda_1^{1/3}$ as the real cubic root of~$\lambda_1$, and for ${i=2,3}$ we define  
${\lambda_i^{1/3}:=\lambda_1^{1/3}\theta_{i1}}$, where ${\theta_{ij}\in \L}$ are defined in item~\ref{icubes} of Proposition~\ref{prlami}. 

Note that ${\theta_{31}=\overline{\theta_{21}}}$. Indeed, ${\overline{\theta_{21}}^3= \overline{\lambda_2}/\overline{\lambda_1}=\lambda_3/\lambda_1}$. Since $\lambda_3/\lambda_1$ has only one cubic root in~$\L$, we must have ${\theta_{31}=\overline{\theta_{21}}}$.

From our definitions it follows that ${\lambda_1^a \lambda_2^a\lambda_3^a}$ is a positive real number. Since ${(\lambda_1^a \lambda_2^a\lambda_3^a)^3= (\lambda_1 \lambda_2\lambda_3)^n=1}$, this proves~\eqref{eprodd}, so we are only left with~\eqref{esumm}.

As we have seen in the proof of Proposition~\ref{prdenom}, there exists a root of unity~$\eta$ such that~\eqref{etrztrib} holds. This can be rewritten as 
${\alpha_1+\alpha_2\theta_{21}^n\eta+\alpha_3\theta_{31}^n\overline{\eta}=0}$. Lemma~\ref{lstrange} implies that 
$
{\eta\in \Q(\alpha_1,\alpha_2,\alpha_3,\theta_{21},\theta_{31})=\L} 
$. 
Since~$\L$ contains no roots of unity other than ${\pm1}$, we must have ${\eta=1}$ or ${\eta=-1}$. In the former case we are done. Now let us assume that ${\eta=-1}$. In this case 
\begin{equation}
\label{etamone}
\alpha_1-\alpha_2\theta_{21}^n-\alpha_3\theta_{31}^n=0.
\end{equation}
Let ${\sigma \in \Gal(\L/\Q)}$ be such that  
$$
\lambda_1^\sigma=\lambda_2, \qquad \lambda_2^\sigma=\lambda_1, \qquad  \lambda_3^\sigma=\lambda_3. 
$$
Then 
$$
\alpha_1^\sigma=\alpha_2,\quad  \alpha_2^\sigma= \alpha_1, \quad \alpha_3^\sigma=\alpha_3, \quad \theta_{21}^\sigma=\theta_{12}=\theta_{21}^{-1}, \quad \theta_{31}^\sigma= \theta_{32}=\theta_{31}\theta_{21}^{-1}. 
$$
Applying~$\sigma$ to~\eqref{etamone}, we obtain 
${\alpha_2-\alpha_1\theta_{21}^{-n}-\alpha_3\theta_{31}^n\theta_{21}^{-n}=0}$, which can be rewritten as 
${\alpha_1-\alpha_2\theta_{21}^n+\alpha_3\theta_{31}^n=0}$. 
Comparing this with~\eqref{etamone}, we obtain ${\alpha_3\theta_{31}^n=0}$, a contradiction. The proposition is proved. 
\end{proof}

\subsection{Proof of Theorem~\ref{thtribo}}
Let~$W$ be the $\Q$-valued LRS with the general term given by 
$$
W(n)= 44\alpha_1^3\lambda_1^n+44\alpha_2^3\lambda_2^n+44\alpha_3^3\lambda_3^n-3.  
$$
It is an LRS of order~$4$, defined by 
$$
W(0)=W(1)=0, \quad W(2)=2, \quad W(3)=8, \qquad W(n+4)= 2W(n+3)- W(n). 
$$

\begin{proposition}
\label{prw}
Let ${a\in \Q}$ be a TRZ of the Tribonacci LRS. Then~$3a$ is a zero of~$W$. 
\end{proposition}

\begin{proof}
As we have seen in Proposition~\ref{prdenom}, if~$a$ is a TRZ of the Tribonacci LRS, then ${n:=3a\in \Z}$, so we only need to prove that ${W(n)=0}$.

Define the rational powers $\lambda_i^a$   as in Proposition~\ref{prnoroots}, so that both~\eqref{esumm} and~\eqref{eprodd} hold. Using~\eqref{eff}, we also find 
\begin{equation}
\label{eprodal}
\alpha_1\alpha_2\alpha_3=\frac{\lambda_1\lambda_2\lambda_3}{P'(\lambda_1)P'(\lambda_2)P'(\lambda_3)}=\frac{1}{44}. 
\end{equation}
Consider the polynomial 
$$
F(X_1,X_2,X_3)=X_1^3+X_2^3+X_3^3-3X_1X_2X_3 \in \Z[X_1,X_2,X_3]. 
$$
Then 
$
{44F(\alpha_1\lambda_1^a, \alpha_2\lambda_2^a,\alpha_3\lambda_3^a)= W(n)} 
$, 
because 
${\alpha_1\lambda_1^a \alpha_2\lambda_2^a\alpha_3\lambda_3^a=  1/44}$, as follows from~\eqref{eprodd} and~\eqref{eprodal}.
The polynomial~$F$  
factors as 
$$
F(X_1,X_2,X_3)=(X_1+X_2+X_3)(X_1+\zeta_3 X_2+\overline{\zeta_3} X_3)(X_1+\overline{\zeta_3} X_2+\zeta_3 X_3),  
$$
which implies that ${F(\alpha_1\lambda_1^a, \alpha_2\lambda_2^a,\alpha_3\lambda_3^a)=0}$ by~\eqref{esumm}. This completes the proof.
\end{proof}

\begin{proposition}
\label{prserw}
The only zeros of~$W$ are ${-51,-12,-5,-3,0,1}$. 
\end{proposition}

\begin{proof}
  In~\cite{BLNOPW22} an algorithm is suggested which, when terminates,
  produces the full set of zeros of a given non-degenerate LRS,
  together with a mathematically rigorous proof that no other zeros
  exist. This algorithm is implemented, for simple\footnote{An LRS is
    called \textit{simple} if its characteristic polynomial has only
    simple roots.} non-degenerate $\Q$-valued LRS, in the
  \textit{Skolem Tool}~\cite{SkolemTool}. Running the Skolem Tool for
  the LRS~$W$, we obtain the result.
\end{proof}

We know (see \cite[Section~2]{BLNOW23}) that $-17,-4,-5/3,-1,0,1/3$ are indeed TRZs of the Tribonacci LRS\@. Hence Theorem~\ref{thtribo} is an immediate consequence of Propositions~\ref{prw} and~\ref{prserw}.

\section{On Question~\ref{quhard}}
\label{squest}
In this section we discuss Question~\ref{quhard}. We will see that the answer is positive for  twisted (integral) zeros of LRS of order~$2$, but (in general) not  for higher order LRS\@. Unless the contrary is stated explicitly, the letter~$p$ in this section denotes a prime number. 

For  twisted (integral) zeros of LRS of order~$2$ we not only answer Question~\ref{quhard}, but obtain a partial analog of Theorem~\ref{theasyk}. 

\begin{theorem}
\label{thquest}
Let~$U$ be a non-degenerate LRS of order~$2$ with values in a number field~$\K$ and ${a\in \Z}$  a twisted zero of~$U$. Then for infinitely many primes~$\gerp$ of~$\K$ the following holds: there exist 
\begin{equation}
\label{eqapkt}
Q\in \Z_{>0}, \qquad a'\in \{0,1, \ldots, Q-1\}, 
\qquad \tau\in \Z
\end{equation}
such that for every ${n\in \Z}$ satisfying ${n\equiv a' \pmod Q}$ we have 
$$
\nu_\gerp(U(n))=\nu_p(n-a)+\tau, 
$$ 
where~$p$ is the rational prime below~$\gerp$. Moreover, ${p\nmid Q}$; in particular, ${\nu_p(n-a)}$ is unbounded on the set of~$n$ satisfying ${n\equiv a'\pmod Q}$. 
\end{theorem}

In fact, we show that this holds true for primes~$\gerp$ from a set of positive lower density. Let us recall the definition of density. Denote by $\pi_\K(x)$ the counting function for primes of~$\K$; that is, the number of primes~$\gerp$ of~$\K$ such that ${\NN\gerp\le x}$; here $\NN\gerp$ denotes the absolute norm. Let~$\PP$ be a set of primes of~$\K$. We denote its lower density as 
$
\liminf_{x\to+\infty}{\#\{\gerp\in \PP: \NN\gerp\le x\}}/{\pi_\K(x)}. 
$

\subsection{A density result}

The proof of Theorem~\ref{thquest} relies on a certain Chebotaryov\footnote{We prefer the spelling \textit{Chebotaryov}, as in~\cite{wikicheb}, because it is more consistent with the original Russian and Ukrainian pronunciation.}-style density result. To state it, let us introduce some more notation. In this subsection~$\K$ is a number field, unless stated otherwise. 

Let  ${\alpha \in \K^\times}$ and a prime~$\gerp$ of~$\K$ be such that ${\nu_\gerp(\alpha)=0}$. We denote by  $\ord_\gerp(\alpha)$ the multiplicative order of~$\alpha$ modulo~$\gerp$. That is, let 
$
{\OO_\gerp:=\{x\in \K: \nu_p(x)\ge 0\}}
$
be the local ring of~$\gerp$, and ${\OO_\gerp \to \OO_\gerp/\gerp: x\mapsto\overline{x}}$ the reduction map\footnote{This is in conflict with the convention of Section~\ref{stribo}, where ${x\mapsto\overline{x}}$ denotes the complex conjugation. However, in this section we never use complex conjugation, so there is no risk of confusion.} modulo~$\gerp$. Then ${\overline{\alpha} \in (\OO_\gerp/\gerp)^\times}$, and $\ord_\gerp(\alpha)$ is the order of~$\overline{\alpha}$ in the multiplicative group $(\OO\gerp/\gerp)^\times$. 

For a positive integer~$r$ we denote by $\PP_\K(\alpha, r)$   the set of $\K$-primes~$\gerp$ such that the multiplicative order $\ord_\gerp(\alpha)$ is divisible by~$r$. In symbols: 
$$
\PP_\K(\alpha, r):=\{\gerp: r\mid \ord_\gerp(\alpha)\}. 
$$

\begin{proposition}
\label{prcheb}
Let~$r$ be a positive integer with the property
\begin{equation}
\label{ezep}
\text{$\zeta_p\in \K$ for every ${p\mid r}$}.
\end{equation}
Let ${\alpha \in \K^\times}$ be not a root of unity. Then the set $\PP_\K(\alpha, r)$ is infinite, and, moreover, it is of positive lower density. 
\end{proposition}

The proof of Proposition~\ref{prcheb} depends on the following lemma. 

\begin{lemma}
\label{lkum}
Let~$p$ be a prime number and~$\K$ a field of characteristic distinct from~$p$.  Define~$\ell$ as the biggest integer such that ${\zeta_{p^\ell}\in \K}$. Assume that ${\ell\ge 1}$. 
Let ${\alpha\in \K}$ be such that ${\alpha^{1/p}\in \K(\zeta_{p^{\ell+1}})}$. Then ${\alpha \in \mu_{p^\ell}\K^p}$. 
\end{lemma}

\begin{proof}
We may assume that ${\alpha\ne 0}$, since there is nothing to prove otherwise.
We use Kummer's Theory, as in Theorem~8.1 from \cite[Chapter~VI]{La02}. Let~$B$ be the subgroup of the multiplicative group~$\K^\times$ generated by~$\alpha$,~$\mu_{p^\ell}$ and $(\K^\times)^p$. By the hypothesis,  ${\K(B^{1/p})=\K(\zeta_{p^{\ell+1}})}$. 
The above-mentioned theorem implies that 
$$
[B:(\K^\times)^p]=[\K(B^{1/p}):\K]=[\K(\zeta_{p^{\ell+1}}):\K]=p. 
$$
Since ${[\mu_{p^\ell}(\K^\times)^p:(\K^\times)^p]=p}$, this proves that ${B=\mu_{p^\ell}(\K^\times)^p}$, which exactly means that ${\alpha \in \mu_{p^\ell}(\K^\times)^p}$. 
\end{proof}

\begin{proof}[Proof of Proposition~\ref{prcheb}]
If the statement holds true with~$\K$ replaced by a bigger field, then it is true for~$\K$. Indeed, let~$\K'$ be a finite extension of~$\K$. If~$\gerp$ is a $\K$-prime, and~$\gerp'$ is a $\K'$-prime above~$\gerp$, then for any ${\alpha\in \K^\times}$ we have ${\ord_{\gerp'}(\alpha)=\ord_\gerp(\alpha)}$; so, it suffices to to show that the set $\PP_{\K'}(\alpha,r)$ is of positive lower density. Thus, we may assume that ${\sqrt{-1}\in \K}$.

Let~$m$ be the order of the group of roots of unity~$\mu_\K$. Condition~\eqref{ezep} may be re-stated as ${p\mid r\Rightarrow p\mid m}$. 

If the conclusion of the proposition holds true with~$r$ replaced a multiple of~$r$, then it holds for~$r$. Hence we may replace~$r$ by   ${\prod_{p\mid r}p^{\max\{\nu_p(r), \nu_p(m)+1\}}}$  and assume in the sequel that 
\begin{equation}
\label{emsqmidr}
\nu_p(r)>\nu_p(m) \qquad (p\mid m); 
\end{equation}
in particular, 
${p\mid r\Leftrightarrow p\mid m}$.

Next, we may replace~$\alpha$ by $\alpha\xi$, where ${\xi\in \K}$ is a root of unity. Indeed, if ${r\mid \ord_\gerp(\alpha)}$, then  ${\nu_p\bigl(\ord_\gerp(\xi)\bigr)< \nu_p\bigl(\ord_\gerp(\alpha)\bigr)}$ for every ${p\mid \ord_\gerp(\xi)}$, by~\eqref{emsqmidr}. Hence ${\ord_\gerp(\alpha\xi)=\ord_\gerp(\alpha)}$.

Finally, we may assume that  
\begin{equation}
\label{eanmukp}
\alpha \notin \mu_\K\K^p  \qquad (p\mid r). 
\end{equation} 
Indeed, call~$\alpha$ with property~\eqref{eanmukp} \textit{$r$-reduced}. Since~$\alpha$ is not a root of unity, we have  ${\alpha\xi=\beta^N}$, where~$\beta$ is $r$-reduced,~$N$ is composed of primes dividing~$r$ and~$\xi$ is a root of unity.  Now, if ${Nr\mid \ord_\gerp(\beta)}$ then ${r\mid \ord_\gerp(\alpha\xi)=\ord_\gerp(\alpha)}$. Hence we may assume~\eqref{eanmukp}, replacing~$\alpha$ by~$\beta$ and~$r$ by $Nr$. 

Denote ${\L:=\K(\zeta_r)}$. Properties~\eqref{ezep} and~\eqref{emsqmidr} imply that 
$$
\Gal(\L/\K) \cong \prod_{p\mid m}\Z/p^{\nu_p(r/m)}\Z . 
$$ 
In particular, for every ${p\mid r}$, the extension~$\L$ has exactly one subfield of degree~$p$ over~$\K$; precisely, it is ${\K(\zeta_{p^{\ell_p+1}})}$, where ${\ell_p:=\nu_p(m)}$. 

Pick some value of the $r^\tho$ root $\alpha^{1/r}$. This would define the roots $\alpha^{1/p}$ for every ${p\mid r}$. 
We claim that ${\alpha^{1/p}\notin \L}$ for any ${p\mid r}$. Indeed, in the opposite case, $\K(\alpha^{1/p})$ would be a subfield of~$\L$ of degree~$p$ over~$\K$. This would imply ${\alpha^{1/p}\in \K(\zeta_{p^{\ell_p+1}})}$, which, by Lemma~\ref{lkum}, implies that ${\alpha \in \mu_{p^{\ell_p}}\K^p}$, contradicting~\eqref{eanmukp}.

Lemma~\ref{llang}, together with Remark~\ref{rlang},  implies now  that ${\M:=\L(\alpha^{1/r})}$ is an extension of~$\L$ of degree~$r$. Moreover, since ${\zeta_r\in \L}$, extension $\M/\L$ is cyclic, and the map 
\begin{equation}
\label{eisom}
\sigma\mapsto\xi_\sigma :\frac{(\alpha^{1/r})^\sigma}{\alpha^{1/r}}
\end{equation}
defines an isomorphism of the groups ${H:=\Gal(\M/\L)}$ and~$\mu_r$. 

Since ${\L=\K(\zeta_r)}$, extension $\M/\K$ is Galois, and we denote  ${G:=\Gal(\M/\K)}$.  Then~$H$ is the subgroup of~$G$ fixing~$\zeta_r$.

Let~$\gerp$ be a $\K$-prime  not dividing~$r$ and satisfying ${\nu_\gerp(\alpha)=0}$. In particular,~$\gerp$ does not ramify in~$\M$. For an $\M$-prime~$\Gerp$  above~$\gerp$, we  denote by ${\phi_\Gerp\in G}$ the Frobenius of~$\Gerp$ above~$\K$. Recall that~$\phi_\Gerp$ is the element of~$G$ with the following property: let 
${\OO_\Gerp:=\{x\in \M:\nu_\Gerp(x)\ge 0\}}$
be the local ring of~$\Gerp$; then for every ${x\in \OO_\Gerp}$ we have ${x^{\NN_\gerp}\equiv x^{\phi_\Gerp}\pmod\Gerp}$ (as before, $\NN(\cdot)$ denotes the absolute norm). 

Next, let ${\bigl((\M/\K)\big/\gerp\bigr):=\{\phi_\Gerp: \Gerp\mid \gerp\}}$ be the Artin symbol of~$\gerp$. 
Note that it is a full conjugacy class in~$G$. 
Denote by~$\Sigma$ the subset of~$H$ consisting of the elements  of exact order~$r$. In symbols:
${\Sigma:=\{\sigma \in H: H=\langle\sigma\rangle \} 
}$.
Let~$\gerp$ be such that ${\bigl((\M/\K)\big/\gerp\bigr)\subset \Sigma}$. By the Theorem of Chebotaryov, the set of such~$\gerp$ is of positive density\footnote{Actually, it is of density $\#\Sigma/\#G$, but the exact value of the density is not relevant to us.}, so we only have to prove that ${r\mid \ord_\gerp(\alpha)}$.

We are going to prove that ${r\mid \NN\gerp-1}$, but~$\overline{\alpha}$ is not a $p^\tho$ power in $(\OO_\gerp/\gerp)^\times$ for any ${p\mid r}$. Since ${\NN\gerp-1}$ is the order of the cyclic group $(\OO_\gerp/\gerp)^\times$, the order of~$\overline{\alpha}$ in this group must be divisible by~$r$, as wanted.

As before, let~$\Gerp$  be an $\M$-prime above~$\gerp$.  Since ${\phi_\Gerp\in \Sigma\subset   H}$, we have ${\phi_\Gerp(\zeta_r)=\zeta_r}$, which means that ${\zeta_r^{\NN\gerp}\equiv \zeta_r\pmod\Gerp}$. Since ${\Gerp\nmid r}$, this implies that ${\zeta_r^{\NN\gerp}= \zeta_r}$, that is, ${r\mid \NN\gerp-1}$. 

We are left with proving that~$\overline{\alpha}$ is not a $p^\tho$ power in $(\OO_\gerp/\gerp)^\times$ for any prime~$p$ dividing~$r$. Fix such~$p$. Since ${\phi_\Gerp\in \Sigma}$, it is of exact order~$r$ in~$H$. Hence ${\xi:=\xi_{\phi_\Gerp}}$ (as defined in~\eqref{eisom}) is a primitive $r^\tho$ root of unity, and ${\eta:=\xi^{r/p}}$ is a primitive $p^\tho$ root of unity.

Applying~\eqref{eisom} with ${\sigma=\phi_\Gerp}$, we obtain ${\alpha^{(\NN\gerp-1)/r}\equiv \xi\pmod\Gerp}$. Raising this congruence to the power $r/p$, we obtain ${\alpha^{(\NN\gerp-1)/p}\equiv \eta\pmod\Gerp}$. Since both sides of this congruence belong to~$\K$, it is actually a congruence modulo~$\gerp$, and it implies the identity ${\overline{\alpha}^{(\NN\gerp-1)/p}= \overline{\eta}}$ in the group $(\OO_\gerp/\gerp)^\times$. If~$\overline{\alpha}$ is a $p^\tho$ power in $(\OO_\gerp/\gerp)^\times$ then ${\overline{\alpha}^{(\NN\gerp-1)/p}=1}$, which is impossible because~$\overline{\eta}$ is a primitive $p^\tho$ root of unity. This completes the proof of the proposition. 
\end{proof}

\subsection{Proof of Theorem~\ref{thquest}}

Besides Proposition~\ref{prcheb}, the proof of Theorem~\ref{thquest} relies on the following lemma. It  is very classical and goes back to the work of Lucas~\cite{lu78} or even earlier. Still, we give a proof for convenience.

\begin{lemma}
\label{lold}
Let~$\K$ be a number field,~$\gerp$ a prime of~$\K$ with underlying rational prime~$p$ and ${e:=\nu_\gerp(p)}$ the ramification index. Let ${\theta\in \K}$ be such that ${\nu_\gerp(\theta-1)>0}$ and ${n\in \Z}$. Then we have the following.

\begin{enumerate}
\item
\label{incoprime}
If ${p\nmid n}$ then ${\nu_\gerp(\theta^n-1)=\nu_\gerp(\theta-1)}$. 

\item
\label{igen}
In general, ${\nu_\gerp(\theta^n-1) \ge \nu_\gerp(\theta-1)+\nu_p(n)\min \{e, (p-1)\nu_\gerp(\theta-1)\}}$. 

\item
\label{ieq}
If ${\nu_\gerp(\theta-1) >e/(p-1)}$ then ${\nu_\gerp(\theta^n-1) = \nu_\gerp(\theta-1)+e\nu_p(n)}$. 

\end{enumerate}
\end{lemma}

\begin{proof}
Item~\ref{incoprime} must be proved only for  ${n>0}$ and ${n=-1}$. If ${n>0}$ then 
$$
\theta^n-1=(\theta-1)(\theta^{n-1}+\cdots+1).
$$
Since ${\theta\equiv 1 \pmod\gerp}$ in the local ring $\OO_\gerp$,
we have 
$$
\theta^{n-1}+\cdots+1\equiv n \pmod\gerp. 
$$ 
In particular, ${\nu_\gerp(\theta^{n-1}+\cdots+1)=0}$ if ${p\nmid n}$. This proves item~\ref{incoprime} for ${n>0}$. As for ${n=-1}$, it is obvious that ${\nu_\gerp(\theta^{-1}-1)=\nu_\gerp(\theta-1)}$. 

Due to item~\ref{incoprime}, in items~\ref{igen} and~\ref{ieq} we may assume that ${n=p^k}$. Moreover, using induction in~$k$, we may assume that ${n=p}$. Write ${\theta=1+\gamma}$. Then
$$
\theta^p-1=p\gamma+\sum_{\ell=2}^{p-1}\binom p\ell \gamma^\ell + \gamma^p. 
$$  
Each of the terms inside the sum has $\gerp$-adic valuation strictly bigger than ${\nu_\gerp(p\gamma)=\nu_\gerp(\gamma)+e}$. Hence 
\begin{equation}
\label{einequality}
\nu_\gerp(\theta^p-1) \ge \nu_\gerp(\gamma) + \min\{e, (p-1)\nu_\gerp(\gamma)\} 
\end{equation}
which proves item~\ref{igen}. Finally, when ${\nu_\gerp(\gamma) >e/(p-1)}$, inequality~\eqref{einequality} becomes equality, and the minimum on the right is~$e$, which proves item~\ref{ieq}. 
\end{proof}

Now we are ready to prove Theorem~\ref{thquest}. 
\begin{proof}[Proof of Theorem~\ref{thquest}]
By shifting, we may  assume that ${a=0}$. Thus, we have to prove that, for infinitely many $\K$-primes~$\gerp$ the following holds: there exist~$Q$,~$a'$ and~$\tau$ as in the statement of the theorem such that 
\begin{equation}
\label{enugerp}
\nu_\gerp(U(n))=\nu_\gerp(n)+\tau \qquad \text{when} \quad n\equiv a' \pmod Q. 
\end{equation}

If the statement holds true with~$\K$ replaced by a bigger number field, then it holds for~$\K$. Hence we may assume that the roots of the characteristic polynomial of~$U$ belong to~$\K$. 
Multiplying $U(n)$ by $\beta\theta^n$ with suitable ${\beta, \theta \in \K^\times}$, we may assume that either ${U(n)=n-\beta}$, where ${\beta\in \K}$, or ${U(n)=\eta\lambda^n -1}$, where ${\eta, \lambda \in K^\times}$. If ${U(n)=n-\beta}$ and~$0$ is a twisted zero of~$U$, then ${\beta=0}$ and ${U(n)=n}$, so there is nothing to prove. 

Now let us assume that ${U(n)=\eta\lambda^n -1}$. Note that~$\lambda$ is not a root of unity, because~$U$ is non-degenerate. Since~$0$ is a twisted zero of~$U$, we have ${\eta\lambda^0-\xi=0}$ for some root of unity~$\xi$. Hence~$\eta$ is a root of unity. Denote by~$r$ its order. 


Let $\PP'_\K(\lambda, r)$ be the subset of $\PP_\K(\lambda, r)$, consisting of ${\gerp\in \PP_\K(\lambda, r)}$, unramified and of degree~$1$ over~$\Q$, and with underlying prime ${\NN\gerp>2}$.
Proposition~\ref{prcheb} implies that the set $\PP_\K(\lambda, r)$ is of positive lower density. Hence so is $\PP'_\K(\lambda, r)$:  this is because the set of unramified primes of degree~$1$ is of density~$1$. We claim that for every ${\gerp\in \PP'_\K(\lambda, r)}$, there exist~$Q$,~$a'$  and~$\tau$ as in~\eqref{eqapkt} 
such that~\eqref{enugerp} holds, and ${p\nmid Q}$. 

Thus, fix ${\gerp\in \PP'_\K(\lambda, r)}$. Since~$\gerp$ is of degree~$1$, the cyclic group ${(\OO_\gerp/\gerp)^\times}$ is  of order ${p-1}$, where ${p:=\NN\gerp}$ is an odd prime number.  Since ${r\mid \ord_\gerp(\lambda)}$, the subgroup $\langle\overline{\lambda}\rangle$ of ${(\OO_\gerp/\gerp)^\times}$ contains~$\overline{\eta}$. Hence there exists ${s\in \{1, \ldots, p-1\}}$ such that ${\eta\lambda^s\equiv 1 \pmod\gerp}$.
Lemma~\ref{lold} implies that 
$$
\nu_\gerp\bigl((\eta\lambda^s)^m-1\bigr) =\nu_p(m) +\nu_\gerp(\xi\lambda^s-1) \qquad (m\in \Z). 
$$
Now note that, when ${m\equiv 1 \pmod r}$, we have ${(\xi\lambda^s)^m-1=U(sm)}$. Hence~\eqref{enugerp} holds with 
$$
Q:=rs, \qquad a':=s, \qquad \tau:=\nu_\gerp(\xi\lambda^s-1). 
$$
Note also that ${p\nmid Q}$ because ${r\mid p-1}$ and ${1\le s\le p-1}$. 
The theorem is proved. 
\end{proof}

\begin{remark}
\begin{enumerate}
\item
Our argument can be illustrated with the LRS from Example~\ref{extwonplusone}. In that case
$$
\K=\Q, \qquad \lambda=2, \qquad \eta =-1, \qquad r=2, \qquad s=\frac{p-1}{2},  
$$
and the set $\PP'_\Q(2, 2)$ contains all the primes satisfying ${p\equiv \pm3\pmod8}$ (and some other primes as well).

\item
This argument does not extend to rational~$a$, because we can no longer do the shifting and assume that ${a=0}$. We do not know whether Theorem~\ref{thquest} can be extended to twisted rational zeros.  

\item
The non-degeneracy hypothesis cannot be dropped. Indeed, consider
$$
U(n):= \zeta_3^n+\overline{\zeta_3}^n =
\begin{cases}
2, & \text{if $3\mid n$}, \\
-1, & \text{if $3\nmid n$}. 
\end{cases}
$$
Then~$0$ is a twisted zero of~$U$, but ${\nu_p(U(n))=0}$ for all~$n$ and all ${p\ne 2}$. 
\end{enumerate}
\end{remark}

\subsection{Concluding remarks}

As we already indicated in the introduction, the answer to Question~\ref{quhard} is negative for LRS of order~$3$ or higher. Here is an example of order~$3$, but one can easily construct similar examples of any order.

\begin{example}
Let~$p$ be a prime number. 
The LRS ${U(n):= 8^n+2^n+1}$  has a twisted zero at~$0$, because ${1+\zeta_3+\overline{\zeta_3}=0}$. However, there does not exist a sequence $(n_k)$ such that 
$$
\nu_p(U(n_k))\to +\infty, \qquad \nu_p(n_k) \to +\infty. 
$$
Indeed, let $(n_k)$ be such a sequence. Let~$\K$ be the splitting field of the polynomial ${X^3+X+1}$ and~$\gerp$ a prime of~$\K$ above~$p$. Then, replacing $(n_k)$ by a subsequence, we have 
$$
\nu_\gerp(2^{n_k}-\alpha)\to +\infty, \qquad \nu_\gerp(n_k) \to +\infty 
$$
for some root~$\alpha$ of this polynomial. Theorem~\ref{thtwratk} implies that~$0$ is a twisted zero of the LRS ${2^n-\alpha}$, a contradiction, because~$\alpha$ is not a root of unity. 
\end{example}

To conclude, let us ask one more question. Theorem~\ref{thtwrat} deals with just one individual prime number. What happens if a sequence $(n_k)$ satisfying~\eqref{enk} can be found for all but finitely many primes? One may expect that in this case~$a$ is a genuine zero of~$U$, not merely a TRZ. 

\begin{question}
Let~$U$ be an LRS with values in a number field~$\K$, and ${a\in \Q}$. 
Assume that for every $\K$-prime~$\gerp$, with finitely many exceptions, there exist a sequence of integers $(n_k)$ (depending on~$\gerp$)
such that
$$
\nu_\gerp(U(n_k))\to +\infty, \qquad \nu_\gerp(n_k-a) \to +\infty. 
$$
Does it imply that either~$a$ a trivial TRZ (as defined in Subsection~\ref{ssq}) or ${a\in \Z}$  and ${U(a)=0}$? 
\end{question}

Probably, ``all but finitely many primes'' can be replaced by ``primes from a set of density~$1$''. But, as Example~\ref{extwonplusone} shows, it is not enough to assume that this holds for infinitely many primes, or even for primes from a set of positive lower density. 

\subsection*{Acknowledgments} 
Yuri Bilu and Florian Luca were supported by the ANR project JINVARIANT. James Worrell was supported by UKRI Fellowship EP/X033813/1. 

While working on this project, Yuri Bilu enjoyed hospitality and support of MPIM Bonn, where he was visiting in June 2023.

Florian Luca worked on this paper in 2023 during research visits at the MPI-SWS, Saarbrücken, Germany and the Stellenbosch Institute for Advanced Studies, Stellenbosch, South Africa. He thanks these institutions for their hospitality and support.

Jo\"el Ouaknine is also affiliated with Keble College,
Oxford as \url{emmy.network} Fellow, and supported by DFG grant 389792660 as
part of TRR 248 (see \url{https://perspicuous-computing.science}).

{\footnotesize

\bibliographystyle{amsplain}
\bibliography{twisted}

\providecommand{\bysame}{\leavevmode\hbox to3em{\hrulefill}\thinspace}
\providecommand{\MR}{\relax\ifhmode\unskip\space\fi MR }
\providecommand{\MRhref}[2]{%
  \href{http://www.ams.org/mathscinet-getitem?mr=#1}{#2}
}
\providecommand{\href}[2]{#2}
\begin{thebibliography}{10}

\bibitem{Ba65}
H.~Bass, \emph{A remark on an arithmetic theorem of {C}hevalley}, Proc. Amer.
  Math. Soc. \textbf{16} (1965), 875--878. \MR{184925}

\bibitem{Bi23}
Yuri Bilu, \emph{The {C}hevalley-{B}ass {T}heorem}, Essays in Analytic Number
  Theory In Honor of Helmut Maier’s 70th Birthday, to appear,
  arXiv:2305.05041 (2023).

\bibitem{BLNOPW22}
Yuri Bilu, Florian Luca, Joris Nieuwveld, Jo\"{e}l Ouaknine, David Purser, and
  James Worrell, \emph{Skolem meets {S}chanuel}, 47th {I}nternational
  {S}ymposium on {M}athematical {F}oundations of {C}omputer {S}cience, LIPIcs.
  Leibniz Int. Proc. Inform., vol. 241, Schloss Dagstuhl. Leibniz-Zent.
  Inform., Wadern, 2022, pp.~Art. No. 20, 15. \MR{4481938}

\bibitem{SkolemTool}
Yuri Bilu, Florian Luca, Joris Nieuwveld, Jo{\"{e}}l Ouaknine, David Purser,
  and James Worrell, \emph{The {S}kolem {T}ool},
  \url{\https://skolem.mpi-sws.org}, 2022, [Online; accessed 15-November-2023].

\bibitem{BLNOW23}
Yuri Bilu, Florian Luca, Joris Nieuwveld, J{ö}el Ouaknine, and James Worrell,
  \emph{On the p-adic zeros of the {T}ribonacci sequence}, Math. Comp. (2023).

\bibitem{BG06}
Enrico Bombieri and Walter Gubler, \emph{Heights in {D}iophantine geometry},
  New Mathematical Monographs, vol.~4, Cambridge University Press, Cambridge,
  2006. \MR{2216774}

\bibitem{Ca76}
J.~W.~S. Cassels, \emph{An embedding theorem for fields}, Bull. Austral. Math.
  Soc. \textbf{14} (1976), 193--198, 479--480. \MR{422221}

\bibitem{Ca86}
\bysame, \emph{Local fields}, London Mathematical Society Student Texts,
  vol.~3, Cambridge University Press, Cambridge, 1986. \MR{861410}

\bibitem{Ch51}
Claude Chevalley, \emph{Deux th\'{e}or\`emes d'arithm\'{e}tique}, J. Math. Soc.
  Japan \textbf{3} (1951), 36--44. \MR{44570}

\bibitem{CJ76}
J.~H. Conway and A.~J. Jones, \emph{Trigonometric {D}iophantine equations ({O}n
  vanishing sums of roots of unity)}, Acta Arith. \textbf{30} (1976), no.~3,
  229--240. \MR{422149}

\bibitem{DZ00}
Roberto Dvornicich and Umberto Zannier, \emph{On sums of roots of unity},
  Monatsh. Math. \textbf{129} (2000), no.~2, 97--108. \MR{1742911}

\bibitem{Go20}
Fernando~Q. Gouv\^{e}a, \emph{{$p$}-adic numbers}, Universitext, Springer,
  Cham, 2020. \MR{4175370}

\bibitem{La02}
Serge Lang, \emph{Algebra}, third ed., Graduate Texts in Mathematics, vol. 211,
  Springer-Verlag, New York, 2002. \MR{1878556}

\bibitem{La84}
Michel Laurent, \emph{\'{E}quations diophantiennes exponentielles}, Invent.
  Math. \textbf{78} (1984), no.~2, 299--327. \MR{767195}

\bibitem{lu78}
Edouard Lucas, \emph{Theorie des {F}onctions {N}umeriques {S}implement
  {P}eriodiques}, Amer. J. Math. \textbf{1} (1878), no.~4, 289--321.
  \MR{1505176}

\bibitem{Ma65}
Henry~B. Mann, \emph{On linear relations between roots of unity}, Mathematika
  \textbf{12} (1965), 107--117. \MR{191892}

\bibitem{ML14}
Diego Marques and Tam\'{a}s Lengyel, \emph{The $2$-adic order of the
  {T}ribonacci numbers and the equation {${T_n=m!}$}}, J. Integer Seq.
  \textbf{17} (2014), no.~10, Article 14.10.1, 8. \MR{3275869}

\bibitem{MT91}
M.~Mignotte and N.~Tzanakis, \emph{Arithmetical study of recurrence sequences},
  Acta Arith. \textbf{57} (1991), no.~4, 357--364. \MR{1109992}

\bibitem{Sc06}
W.~H. Schikhof, \emph{Ultrametric calculus}, Cambridge Studies in Advanced
  Mathematics, vol.~4, Cambridge University Press, Cambridge, 2006, An
  introduction to $p$-adic analysis, Reprint of the 1984 original [MR0791759].
  \MR{2444734}

\bibitem{Sm70}
John~H. Smith, \emph{A result of {B}ass on cyclotomic extension fields}, Proc.
  Amer. Math. Soc. \textbf{24} (1970), 394--395. \MR{257049}

\bibitem{Sm71}
C.~J. Smyth, \emph{On the product of the conjugates outside the unit circle of
  an algebraic integer}, Bull. London Math. Soc. \textbf{3} (1971), 169--175.
  \MR{289451}

\bibitem{pari}
{The PARI~Group}, Univ. Bordeaux, \emph{{PARI/GP version \texttt{2.15.4}}},
  2023, available from \url{http://pari.math.u-bordeaux.fr/}.

\bibitem{wikicheb}
Wikipedia, \emph{Nikolai {C}hebotaryov},
  \url{https://en.wikipedia.org/wiki/Nikolai_Chebotaryov}, 2023, [Online;
  accessed 15-November-2023].

\bibitem{Za95}
U.~Zannier, \emph{Vanishing sums of roots of unity}, Rend. Sem. Mat. Univ.
  Politec. Torino \textbf{53} (1995), no.~4, 487--495, Number theory, II (Rome,
  1995). \MR{1452400}

\end{thebibliography}

}
\end{document}